\documentclass[12pt,twoside,reqno]{amsart}
\usepackage{amsmath}
\usepackage{amsfonts}
\usepackage{amssymb}
\usepackage{color}
\usepackage{mathrsfs}
\usepackage{cite}
\usepackage{geometry}
\newcommand{\modph}{\color{blue}}
\newtheorem{theorem}{Theorem}[section]
\newtheorem{corollary}[theorem]{Corollary}
\newtheorem{lemma}[theorem]{Lemma}
\newtheorem{proposition}[theorem]{Proposition}

\newtheorem{remark}[theorem]{Remark}
\numberwithin{equation}{section}
\begin{document}
\title{Global existence and uniform boundedness in\\ a chemotaxis model with signal-dependent motility}
%\thanks{}
\author{Jie Jiang}
\address{Innovation Academy for Precision Measurement Science and Technology, Chinese Academy of Sciences, Wuhan 430071, HuBei Province, P.R. China}
\email{jiang@apm.ac.cn, jiang@wipm.ac.cn}

\author{Philippe Lauren\c{c}ot}
\address{Institut de Math\'ematiques de Toulouse, UMR~5219, Universit\'e de Toulouse, CNRS \\ F--31062 Toulouse Cedex 9, France}
\email{laurenco@math.univ-toulouse.fr}

\keywords{global existence - boundedness - chemotaxis - comparison}
\subjclass{35B60 - 35K20 - 35K65 - 35M33 - 35Q92}

\date{\today}

%%%%%%%%%%%%%%%%
%%%%%%%%%%%%%%%%
\begin{abstract}
 Global existence is established for classical solutions to a chemotaxis model with signal-dependent motility for a general class of motility functions $\gamma$ which may in particular decay in an arbitrary way at infinity. Assuming further that $\gamma$ is non-increasing and decays sufficiently slowly at infinity, in the sense that $\gamma(s)\sim s^{-k}$ as $s\to\infty$ for some $k\in (0,N/(N-2)_+)$, it is also shown that global solutions are uniformly bounded with respect to time. The admissible decay of $\gamma$ at infinity here is higher than  in previous works.
\end{abstract}
%%%%%%%%%%%%%%%%
%%%%%%%%%%%%%%%%

\maketitle

%
%     HEADLINES
%
\pagestyle{myheadings}
\markboth{\sc{J. Jiang \& Ph. Lauren\c{c}ot}}{\sc{Global existence and uniform boundedness in a chemotaxis model}}

%%%%%%%%%%%%%%%%
%%%%%%%%%%%%%%%%
\section{Introduction}\label{sec1}
%%%%%%%%%%%%%%%%
%%%%%%%%%%%%%%%%

We study the global existence of non-negative classical solutions to the parabolic-elliptic version ($\tau=0$) of the following Keller--Segel chemotaxis model:
\begin{equation}\label{ls}
	\begin{cases}
	\partial_t u =\Delta(u\gamma(v))\,, \qquad (t,x)\in (0,\infty)\times\Omega\,,\\
	\tau \partial_t v =\Delta v-v+u\,, \qquad (t,x)\in (0,\infty)\times\Omega\,,
	\end{cases}
\end{equation}
supplemented with homogeneous Neumann boundary conditions for $u$ and $v$. Here,  $\Omega$ is a bounded domain of $\mathbb{R}^N$ ($N\ge 1$) with smooth boundary, the relaxation time $\tau$ is a non-negative constant, and $u$ and $v$ denote the density of cells and the concentration of a  chemical signal, respectively. The motion of cells is biased by the local concentration of chemotactic signal and is prescribed by the motility $\gamma(v)$ of cells, which is a positive  function of $v$. 

The system~\eqref{ls} is a special version of the chemotaxis model derived by  Keller \& Segel  in their seminal work \cite{1971KS}  based on a random walk theory:
 \begin{equation}\label{ksv}
\begin{cases}
\partial_t u =\mathrm{div}(\mu(v)\nabla u-u\chi(v)\nabla v)\,,\\
\tau \partial_t v =\Delta v-v+u\,,
\end{cases}
\end{equation}
where the cell diffusion rate $\mu$ and the chemotactic sensitivity $\chi$ are assumed to depend only upon the signal concentration. In \eqref{ksv}, the cell diffusion rate $\mu$ and chemotactic sensitivity $\chi$ are linked via
\begin{equation}\label{ksv0}
\chi(v)=(\sigma-1)\mu'(v).
\end{equation}
Here the parameter $\sigma\geq 0$ is the ratio of  the distance between signal receptors over the walk length. If $\sigma>0$, cells can perceive the gradient of concentration by comparing them at different spots and hence chemotactic movement is induced by a gradient sensing mechanism. If $\sigma=0$, there is only one receptor in a cell. Thus an individual cell can only detect the concentration at one spot and the decision is made due to a local sensing mechanism without measuring a chemical gradient. One notices that $\sigma=0$ implies that $\chi=-\mu'$, so that system~\eqref{ksv} can be written in the concise form \eqref{ls} with $\gamma=\mu$. More recently, system~\eqref{ls} with a logistic source term in the first equation was introduced in \cite{2011Science, 2012PRL}  to study the process of pattern formations via the so-called "self-trapping" mechanism. There, $\gamma$ is assumed to be a decreasing function of $v$,  which characterizes a repressive effect of the signal concentration on the cellular motility. 

Recently, the mathematical research on system~\eqref{ls} has  attracted a lot of interest. From the analysis point of view, the main difficulty lies in a possible degeneracy  as $v$ becomes unbounded when a positive and monotone decreasing motility function is involved.  In view of the special Laplacian structure of the first equation in \eqref{ls}, duality and energy techniques were applied to obtain the global existence of solutions to \eqref{ls} in some recent works with certain additional restrictions on $\gamma$. In \cite{TW2017}, the motility $\gamma$ is assumed to range in a closed interval of $(0,\infty)$, thereby excluding the possibility of degeneracy and global classical solutions in two-space dimension, as well as global weak solutions in higher dimensions, were obtained. In \cite{DKTY2019} and \cite{BLT2020}, global weak solutions were constructed for $\gamma(s)=1/(c+s^{k})$, $c\geq0$, when $N\leq 3$, and for $\gamma(s)=e^{-s}$ when $N\geq1$, respectively.

A different approach, relying on a tricky two-steps comparison argument, was recently designed in \cite{FuJi2020,FuJi2020c} to derive directly upper bound estimates on $v$ for generic motility functions (specifically, $\gamma$ is only required to be positive and decay to zero at infinity). More precisely, it was proved that $v$  grows pointwisely at most exponentially in time  in any space dimension, so that degeneracy cannot take place in finite time. Thanks to this feature, global existence and boundedness of classical solutions in any space dimension are established via energy methods and Moser iterations in \cite{FuJi2020c, FuJi2020b, Jiang2020},  provided $\gamma$ behaves in a suitable way at infinity. More recently, through an adaptation of the comparison argument together with a standard compactness argument, the existence of  weak solutions to system~\eqref{ls} in any space dimension with a generic motility is also studied in \cite{LiJi2020}.  Still in arbitrary space dimension but with a completely different method, global existence of weak solutions to \eqref{cp} is also shown in \cite{BLT2020} for the particular choice $\gamma(s)=e^{-s}$, exploiting the gradient structure of \eqref{cp} available in that case.

In this paper, we consider the following initial Neumann boundary value problem to the quasi-stationary version of \eqref{ls}, corresponding to $\tau=0$:
\begin{subequations}\label{cp}
\begin{align}
& \partial_t u  = \Delta (u \gamma(v))\,, \qquad (t,x)\in (0,\infty)\times\Omega\,, \label{cp1} \\
&	- \Delta v + v  =  u\,, \qquad (t,x)\in (0,\infty)\times\Omega\,, \label{cp2} \\
& \nabla (u\gamma(v)) \cdot \mathbf{n} = \nabla v \cdot \mathbf{n} = 0\,, \qquad (t,x)\in (0,\infty)\times\partial\Omega\,, \label{cp3} \\
& u(0)  = u^{in}\,, \qquad x\in\Omega\,. \label{cp4}
\end{align}
\end{subequations}

To motivate the study performed herein, it is worth pointing out that stationary solutions to system~\eqref{cp} are also stationary solutions to the extensively studied parabolic-elliptic Keller--Segel chemotaxis model:
\begin{equation}\label{ks}
	\begin{cases}
	 \partial_t u  = \mathrm{div}(\nabla u-u\nabla \phi(v))\,, &\qquad (t,x)\in (0,\infty)\times\Omega\,, \\
		- \Delta v + v  =  u\,, &\qquad (t,x)\in (0,\infty)\times\Omega\,,
\end{cases}
\end{equation} 
provided that
\begin{equation*}
	\gamma(v)=e^{-\phi(v)}\,.
\end{equation*} 
It is thus of high interest to investigate whether there are some similarities or differences in the dynamical behavior of the solutions to both systems. When $\phi(s)=\chi s$ for some $\chi>0$, \eqref{ks} is usually referred to as the parabolic-elliptic minimal/classical Keller--Segel model and has received considerable attention in the last decades. It is well-known that finite time blowup may happen when $N\ge 2$, see \cite{Biler2020,NS1998,JL1992,HV1996,Nagai1995} and the references therein. In particular, when $N=2$, there is a threshold value $\Lambda_c>0$ of the total mass $\|u^{in}\|_{L^1(\Omega)}$ of the initial condition $u^{in}$ which only depends on $\Omega$ and separates the dynamical behavior of solutions to \eqref{ks}: the solution is global and uniformly-in-time bounded if  the total mass is less than $\Lambda_c$; otherwise, blowup may take place in finite time. When $\phi(s)=\chi s^{\beta}$ for some $\chi>0$ and $\beta>0$ and $\Omega$ is a ball, it is proved in \cite{NS1998} that radially symmetric solutions to \eqref{ks} are global and bounded if $0<\beta<1$ and $N=2$, but that finite time blowup may take place as soon as $\beta>0$  and $N\geq3$.   When $\phi(s)=\chi\log s$, system~\eqref{ks} is called the logarithmic Keller--Segel model and there is some evidence in the literature indicating that the parameter $\chi$ plays a key role in the dynamics of its solutions. It is conjectured in \cite{FuSe2018} that the threshold number separating boundedness and blowup of classical solutions to \eqref{ks} is $N/(N-2)$. However, up to now, the boundedness and blowup results are still far from complete. In the radially symmetric setting, existence of globally bounded solutions to \eqref{ks} is obtained in \cite{NS1998} for $\chi<\infty$ if $N=2$, and for $\chi<2/(N-2)$ if $N\geq3$. Moreover, finite time blowup was constructed for $\chi>2N/(N-2)$ if $N\geq3$. Without the radial symmetry requirement, existence of weak solutions is proved for $\chi\leq 1$ if $N=2$, and $\chi<2/N$ if $N\geq3$ \cite{Biler1999}. Later, in \cite{FWY2015}, globally bounded classical solutions are obtained for $\chi<2/N$ if $N\geq2$. We remark that, when $\chi>1$, it is known that the constant steady state is unstable \cite{1970KS}.

\medskip

Let us next review some related studies on problem~\eqref{cp} in the literature. In \cite{AhnYoon2019}, the existence of a global and bounded classical solution to \eqref{cp} is proved when $\gamma(s)=s^{-k}$ ($\phi(s)=k\log s$, correspondingly) for any $0<k<\infty$ if $N=2$, and for any $0<k<2/(N-2)$ if $N\geq3$, the proof exploiting delicate energy estimates. Later, based on a new comparison argument, it is  proved that classical solutions to \eqref{cp} exist globally for a generic motility function when $N=2$ \cite{FuJi2020}. In addition, these solutions are shown to be uniformly bounded with respect to time, as soon as $\gamma$ decreases at a subcritical rate, in the sense that $\gamma(s)$ decreases slower than $e^{-\alpha s}$ for any $\alpha>0$ as $s\rightarrow\infty$ \cite{FuJi2020b}. This implies in particular that \eqref{cp} has global and bounded classical solutions when $\gamma(s)=e^{-\chi s^{\beta}}$ ($\phi(s)=\chi s^{\beta}$, correspondingly) for any $\chi>0$ and $\beta\in(0,1)$ when $N=2$. Moreover, for the specific case $\gamma(s)=e^{-\chi s}$ ($\phi(s)=\chi s$, correspondingly) with $\chi>0$, a new critical mass phenomenon in two space dimensions is observed: there is a critical value of the total mass $\|u^{in}\|_{L^1(\Omega)}$ below which the global solution to \eqref{cp} is uniformly-in-time bounded while the trajectory $\{ u(t)\ :\ t\ge 0\}$ may be unbounded in $L^\infty(\Omega)$ when $\|u^{in}\|_{L^1(\Omega)}$ exceeds the threshold value \cite{FuJi2020}. The threshold value turns out to be the same as for the minimal Keller--Segel model ($\phi(s)=\chi s$ in \eqref{ks}), but the dynamical behavior is different in the two models, blowup being delayed to infinite time for \eqref{cp} \cite{BLT2020, FuJi2020, JW2020}). In a subsequent work \cite{FuJi2020b}, global existence of classical solutions to \eqref{cp} when $N\geq3$ is established for $\gamma(s)=s^{-k}$ provided that $0<k<(\sqrt{2N}+2)/(N-2)$ and boundedness is verified for $0<k<2/(N-2)$, again by a different method. Later, boundedness is extended to the case $k\in (0,1]$ when $N=4,5$, and $0<k<4/(N-2)$ when $N\geq6$ in \cite{Jiang2020}, based on an observation of an entropy structure and an application of a modified Moser iteration. 

\medskip

In this paper, we aim to improve existence and boundedness results of classical solutions to system~\eqref{cp} based on some new findings. To begin with, we introduce some basic assumptions and notations. Throughout this paper we use the short notation $\|\cdot\|_{p}$ for the norm $\|\cdot\|_{L^p(\Omega)}$ with  $p\in[1,\infty]$. For the initial condition $u^{in}$, we require that
\begin{equation}\label{ini}
u^{in}\in C(\bar{\Omega})\,, \quad u^{in}\geq0\,\, \text{in}\;\bar{\Omega}\,, \quad u^{in}\not\equiv0.
\end{equation}
In view of \eqref{cp2}, we may define the initial value of $v$ by $v^{in}\triangleq (I-\Delta)^{-1}u^{in}$ with $\Delta$ being the usual Laplace operator supplemented with homogeneous Neumann boundary conditions. According to the above definition, standard regularity theory of elliptic equations guarantees that $v^{in}\in W^{2,p}(\Omega)$ for all $p\in(1,\infty)$. For the motility function $\gamma$, we assume that 
\begin{equation}
\begin{cases}
& \gamma\in C^3((0,\infty))\,, \qquad \gamma>0 \;\;\text{in}\;\;(0,\infty)\,, \\
& \\
& K_s \triangleq \displaystyle{\sup_{\tau\in [s,\infty]}\{\gamma(\tau)\}}< \infty \;\;\text{ for all }\;\; s>0\,.
\end{cases} \tag{A0} \label{g0}
\end{equation}
We point out here that sign changes of $\gamma'$ are permitted by assumption \eqref{g0}, a feature that leads to different responses of the cells. Indeed, if $\gamma'<0$, then cells move toward higher chemical concentrations, while there is a chemorepulsive effect on the cells when $\gamma'>0$. 

We are now in a position to state our first main result concerning global existence of classical solutions to \eqref{cp}.

%%%%%%%%%%%%%%%%
\begin{theorem}\label{TH1} Let $N\ge 2$. Suppose that $\gamma$ satisfies assumption \eqref{g0} and the initial condition $u^{in}$ satisfies \eqref{ini}. Then problem \eqref{cp} has a unique global non-negative classical solution $(u,v)\in (C([0,\infty)\times \bar{\Omega}) \cap C^{1,2}((0,\infty)\times \bar{\Omega}))^2$.
\end{theorem}
%%%%%%%%%%%%%%%%

 Observe that \eqref{g0} does not require $\gamma$ to be bounded as $s\to 0$, so that our analysis includes in particular $\gamma(s)=s^{-k}$ for $k>0$. This feature is actually not surprising in view of \eqref{e00} below, which states that $v$ has a time-independent strictly positive lower bound $v_*$, which is also independent of the choice of $\gamma$. Furthermore, the outcome of Theorem~\ref{TH1} confirms that the dynamics of \eqref{cp} differs markedly from that of \eqref{ks}. Indeed, global existence is the rule for \eqref{cp} in any space dimension while, according to the discussion above, it only holds true for \eqref{ks} when $\phi$ satisfies some growth conditions depending on the space dimension $N$.

We next investigate the boundedness of classical solutions to \eqref{cp} in the chemoattractive regime. For this purpose, we propose the following assumptions on $\gamma$:
\begin{equation}
	\gamma\in C^3((0,\infty))\,, \qquad \gamma>0,\;\;\gamma'\leq 0 \;\;\text{in }\;\; (0,\infty)\,, \tag{A1} \label{g1}
\end{equation}
and, moreover,
\begin{equation}\label{gamma2}
\text{there are $k\geq l \geq0$ such that}\;\;\liminf\limits_{s\rightarrow\infty}s^{k}\gamma(s)>0\;\;\text{and}\;\;
\limsup\limits_{s\rightarrow\infty}s^{l}\gamma(s)<\infty. \tag{A2}
\end{equation}
We have the following remarks on the above assumptions on $\gamma$. First, under the assumption \eqref{g1}, $\gamma$ satisfies \eqref{g0} with  $K_s=\gamma(s)$ for $s>0$. Second, the decay assumption \eqref{gamma2} is only assumed as $s\rightarrow\infty$, which allows $\gamma$ to  behave in an arbitrary way within a finite region. Besides,  we point out that  $\limsup\limits_{s\rightarrow\infty}s^{l}\gamma(s)<\infty$ with some $l>0$ and $\gamma'\leq 0$ imply that $\lim\limits_{s\rightarrow\infty}\gamma(s)=0.$ 
Some typical examples of motilities $\gamma$ satisfying \eqref{g1}-\eqref{gamma2} are $(a_1+s)^{-k_1}$, $(a_1+s)^{-k_1}\log^{-k_2} (a_2+s)$, and $(a_1+s)^{-k_1}+(a_2+s)^{-k_2}$, with  $a_i\ge 0$, $k_i>0$ ($i=1,2$).

Our second result on the boundedness of global solutions to \eqref{cp} is stated as follows.

%%%%%%%%%%%%%%%%
\begin{theorem}\label{TH2}
 Let $N\ge 3$. Suppose that $\gamma$ satisfies assumption \eqref{g1}, as well as  assumption \eqref{gamma2} with some $k\geq l \geq0$ satisfying 
\begin{equation*}
 	k < \frac{N}{N-2} \;\;\text{ and }\;\; k-l < \frac{2}{N-2}\,.
 \end{equation*}
  Then, for any initial condition $u^{in}$ satisfying  \eqref{ini}, problem \eqref{cp} has a unique global non-negative classical solution $(u,v)$ which is  uniformly-in-time bounded, i.e., there is a time-independent constant $C>0$ depending only on $u^{in}$, $\gamma$, $N$, and $\Omega$ such that
\begin{equation*}
\|u(t)\|_{\infty}+\|v(t)\|_{W^{1,\infty}}\leq C\,, \qquad t\geq 0\,.
\end{equation*}
\end{theorem}
%%%%%%%%%%%%%%%%

%%%%%%%%%%%%%%%%
\begin{remark}We have the following remarks:
	\begin{itemize}
		\item When $l=0$ and $0<k<2/(N-2)$, Theorem~\ref{TH2} improves \cite[Theorem~1.2]{FuJi2020b} by removing an additional constraint involving $\gamma$, $\gamma'$ and $\gamma''$, which was required to hold pointwisely;
		\item If $\gamma$ satisfies \eqref{g1} and has a strictly positive lower bound on $(0,\infty)$,  the classical solution to \eqref{cp} is uniformly-in-time bounded according to Theorem~\ref{TH2}, which we may apply in that case with $k=l=0$ in assumption \eqref{gamma2}.
	\end{itemize}
\end{remark} 
%%%%%%%%%%%%%%%%

As a consequence of the above results together with the outcome of \cite{AhnYoon2019} when $N=2$ and the exponential stabilization provided in \cite{Jiang2020}, we have the following result for motility functions which are negative power laws.

%%%%%%%%%%%%%%%%
\begin{proposition}Assume $N\geq2$ and  $\gamma(s)=s^{-k}$ for some $k>0$. For any initial condition $u^{in}$ satisfying \eqref{ini}, problem \eqref{cp} has a unique global classical solution. Moreover, it is uniformly-in-time bounded when, either $k<\infty$ and $N=2$, or $k<N/(N-2)$ and $N\geq3$. Furthermore, if $k\leq1$, then the  solution to \eqref{cp} decays exponentially fast in $L^\infty(\Omega)\times W^{1,\infty}(\Omega)$ to the constant steady state $\left( \frac{\| u^{in}\|_1}{|\Omega|},\frac{\| u^{in}\|_1}{|\Omega|} \right)$ as time goes to infinity. 
\end{proposition}
%%%%%%%%%%%%%%%%

We stress that the assumptions on $\gamma$ in this paper are greatly weakened compared with those in previous works  \cite{FuJi2020b,Jiang2020,AhnYoon2019}. On the one hand, for the global existence, only \eqref{g0} is assumed and no monotonicity is required.  In particular, sign-changing of $\gamma'$ is permitted in our setting.  On the other hand, uniform boundedness with respect to time is obtained for motility functions belonging to a much larger set of decreasing functions and the decay rate condition \eqref{gamma2} is also looser than those in  \cite{FuJi2020b,Jiang2020,AhnYoon2019}. In particular, for the specific case $\gamma(s)=s^{-k}$, the admissible range for boundedness is improved to $k\in(0,N/(N-2))$ for all $N\geq 3$. Thus we achieve the uniform boundedness of classical solutions to \eqref{cp} within the whole range $(0,N/(N-2))$, which is also conjectured to ensure uniform boundedness of classical solutions to the logarithmic Keller-Segel model, corresponding to  \eqref{ks} with $\phi(v) = - \log{\gamma(v)} = k \log{v}$ \cite{FuSe2018}.

\medskip

Now let us briefly illustrate the main strategy of our proof. The first step is to establish an upper bound of $v$. Following the comparison argument developed in \cite{FuJi2020,LiJi2020}, one obtains by (formally) taking $(I-\Delta)^{-1}$ on both sides of \eqref{cp1}  the following key identity:
\begin{equation*}
\partial_t v+u\gamma(v)=(I-\Delta)^{-1}[u\gamma(v)].
\end{equation*}
Then, thanks to \eqref{g0} and \eqref{cp2}, one can estimate the nonlocal term by the comparison principle for elliptic equations to deduce that
\begin{equation}\label{comp}
(I-\Delta)^{-1}[u\gamma(v)]\leq (I-\Delta)^{-1}[K_\gamma u] =K_\gamma v
\end{equation} 
with some positive constant $K_\gamma>0$ which will be specified below. Then a time-dependent upper bound for $v$ can be derived by a direct application of Gronwall's inequality. In addition, a uniform-in-time upper bound can be obtained with additional growth assumptions on $\gamma$. Indeed, substituting \eqref{cp2} into the above key identity, we find that
\begin{equation}\label{kid}
\partial_t v -\gamma(v)\Delta v+v\gamma(v)=(I-\Delta)^{-1}[u\gamma(v)].
\end{equation}
In \cite{FuJi2020b,Jiang2020}, applying a delicate Alikakos--Moser type iteration to \eqref{kid}, it is proved that, in order to get a uniform-in-time upper bound of $v$, it suffices to prove a time-independent  estimate on  $\|v\|_{q}$ for some $q>kN/2$. To achieve the latter goal, both works pay efforts to derive the uniform-in-time $H^1$-boundedness of $v$ by means of a duality argument or entropy structure. However, since $H^1(\Omega)\hookrightarrow L^{q}(\Omega)$ for any  $1\leq q\leq 2N/(N-2)$, such a tactic works for $k<4/(N-2)$ at most, and hence apparently, $k$ must be less than $1$ when $N\geq6$ along this idea. We recall that $0<k\leq 1$ is the range for the constant steady state  $\left( \frac{\| u^{in}\|_1}{|\Omega|},\frac{\| u^{in}\|_1}{|\Omega|} \right)$ to be  linearly stable \cite{AhnYoon2019,Jiang2020}.

In the present work, we handle  the nonlocal term on the right hand side of \eqref{kid} in a more technical way. Owing to the comparison principle for elliptic equations and the non-increasing property of $\gamma$, we prove  that
\begin{equation}\label{comp1}
	(I-\Delta)^{-1}[u\gamma(v)]\leq \Gamma(v)+C\,.
\end{equation}
Here, $\Gamma'=\gamma$ and $C$ is a positive constant depending only on $\gamma$ and the initial condition. For the special case $\gamma(s)=s^{-k}$, $\Gamma$ is a power type  function of order $1-k$. Thus the nonlocal term is bounded by a sublinear nonlinearity,  and the growth of the right hand side of \eqref{kid} with \eqref{comp1} is thus of lower order than the one derived from \eqref{comp}. This key observation enables us to establish time-independent  estimates of $\|v\|_{q}$ for any sufficiently large $q$ under the condition $0<k<N/(N-2)$. Then, by a (simplified) Moser iteration, we can derive a uniform-in-time upper bound of $v$ for $k$ within this range.

Once the (time-independent) upper bound obtained for $v$, the existence and boundedness of classical solutions to \eqref{cp} are proved via energy estimates in the already available literature by establishing (time-independent) $L^\infty_t L^p_x$-estimates for $u$ for some $p>N/2$ \cite{Jiang2020,FuJi2020,FuJi2020b,FuJi2020c}. However, such an idea turns out to be rather restrictive since, in all these works,  various inequalities involving $\gamma$, $\gamma'$, and $\gamma''$ are assumed to hold pointwisely, thereby excluding most functions satisfying \eqref{g1} and \eqref{gamma2}. 

In the current work, by transforming  \eqref{kid} to a quasilinear parabolic equation for $v$ in divergence form with a right-hand side featuring a local nonlinearity depending on $|\nabla v|^2$  and a nonlocal nonlinear but regularizing operator, we try to establish an estimate for $\|v\|_{W^{2,p}}$ rather than $\|u\|_{p}$. First, with the help of the upper bound of $v$, we show the H\"older continuity of $v$ with respect to both $t$ and $x$ following a classical approach developed in \cite{LSU1968}. In fact, by an application of a comparison argument to the nonlocal term, a local energy bound for $v$ is derived, which allows us to conclude that $v$ belongs to some H\"older space.  Thanks to this property, we are in a position to employ the theory developed by Amann in \cite{Aman1988, Aman1989, Aman1990, Aman1993, Aman1995} to find a representation formula for $v$ which involves a parabolic evolution operator having similar properties as an analytic semigroup. Thanks to estimates for the parabolic evolution operator given in \cite{Aman1995}, we are finally able to derive the (time-independent) $W^{2,p}$-estimates of $v$ for any $N<p<\infty$, which implies the (time-independent) $L^p$-boundedness of $u$ due to \eqref{cp2}. Our new approach indicates that $L^\infty$-boundedness of $v$ is sufficient to conclude the existence of global classical solutions to \eqref{cp} and moreover, if this upper bound is independent of time, then the solution is uniformly-in-time bounded.

The remainder of the paper is organized as follows.  In Section~\ref{sec2}, we provide some preliminary results and recall some useful lemmas. In Section~\ref{gecs}, we study the global existence of classical solutions to \eqref{cp}. An upper bound on $v$ is first shown via a comparison argument. Then we prove $W^{2,p}$-estimates on $v$ by the theory developed by Amann \cite{Aman1995} for non-autonomous parabolic equations, which is based on analytic semigroups. In Section~\ref{ubcs}, we first establish a time-independent $L^p$-estimate for $v$ for some sufficiently large $p>1$. Then we derive uniform-in-time upper bounds on $v$ by a delicate Alikakos--Moser iteration. Finally, the uniform-in-time boundedness of $u$ is  proved in a similar way as done in Section~\ref{gecs} with minor modifications.

%%%%%%%%%%%%%%%%%%%%
%%%%%%%%%%%%%%%%%%%
\section{Preliminaries}\label{sec2}
%%%%%%%%%%%%%%%%
%%%%%%%%%%%%%%%%

In this section, we recall some useful lemmas which will be used in the subsequent parts. First, we state the existence of local classical solutions which can be proved  in the  same manner as in \cite[Lemma~3.1]{AhnYoon2019} based on standard fixed point argument and  regularity theory for elliptic equations.

%%%%%%%%%%%%%%%%
 \begin{theorem}\label{local}
	Let $\Omega$ be a smooth bounded domain of $\mathbb{R}^N$. Suppose that $\gamma$ satisfies \eqref{g0} and $u^{in}$ satisfies \eqref{ini}. Then there exists $T_{\mathrm{max}} \in (0, \infty]$ such that problem \eqref{cp} has a unique nonnegative classical solution $(u,v)\in (C([0,T_{\mathrm{max}})\times\bar{\Omega})\cap C^{1,2}((0,T_{\mathrm{max}})\times\bar{\Omega}))^2$. The solution $(u,v)$ satisfies the mass conservation:
	\begin{equation}
	\int_\Omega u(t,x)\ \mathrm{d} x=\int_\Omega v(t,x)\ \mathrm{d} x= \int_\Omega u^{in}(x)\ \mathrm{d} x
	\quad \text{for\ all}\ t \in (0,T_{\mathrm{max}}).\label{e0}
	\end{equation}		
	If $T_{\mathrm{max}}<\infty$, then
	\begin{equation*}
	\limsup\limits_{t\nearrow T_{\mathrm{max}}}\|u(t)\|_{\infty}=\infty.
	\end{equation*}
\end{theorem}
%%%%%%%%%%%%%%%%

We also recall that, by \cite[Lemma~2.1]{FWY2015}, {\eqref{cp2}, \eqref{e0}, and the positivity of $\|u^{in}\|_1$,} there is $v_*>0$ depending only on $\Omega$ and $\|u^{in}\|_1$ such that
\begin{equation}
v(t,x) \ge v_*\,, \qquad (t,x)\in (0,T_{\mathrm{max}})\times \Omega\,. \label{e00}
\end{equation}

Next, we let $(\cdot)_+=\max\{\cdot,0\}$ and recall the following result,  see \cite[Proposition~(9.2)]{Aman1983}, \cite[Lemme~3.17]{BeBo1999}, or \cite[Lemma~2.2]{AhnYoon2019}.

%%%%%%%%%%%%%%%%
\begin{lemma}\label{lm2}
	Let $\Omega$ be  a smooth bounded domain in $\mathbb{R}^N$, $N\geq1$, and $f\in  L^1(\Omega)$. For any $1\leq q< \frac{N}{(N-2)_+}$, there exists a positive constant $C=C(N,q,\Omega)$ such that the solution $z \in W^{1,1}(\Omega)$ to
	\begin{equation*}
	\begin{cases}
	-\Delta z+z=f,\qquad x\in\Omega\,,\\
	\nabla z\cdot \mathbf{n}=0\,,\qquad x\in\partial\Omega\,,
	\end{cases}
	\end{equation*}
	 satisfies
	\begin{equation*}
	\|z\|_q \leq C(N,q,\Omega) \|f\|_1\,.
	\end{equation*}
\end{lemma}
%%%%%%%%%%%%%%%%

An immediate consequence of \eqref{e0} and Lemma~\ref{lm2} is the following time-independent bound on  the second component $v$ of the solution to \eqref{cp}.

%%%%%%%%%%%%%%%%
\begin{corollary}\label{cr3}
For $q\in \left( 1,\frac{N}{(N-2)_+} \right)$, there is $C(q)>0$ such that
\begin{equation*}
\|v(t)\|_q \le C(q)\,, \qquad  t \in [0,T_{\mathrm{max}})\,.
\end{equation*}
\end{corollary}
%%%%%%%%%%%%%%%%

We next fix $a\in (0,v_*)$ and define 
\begin{equation*}
\Gamma(s)=\int_{a}^s\gamma(\eta)\mathrm{d}\eta, \qquad s>0.
\end{equation*}
The following elementary lower bound for $\Gamma$ is derived in \cite[Lemma~4.2]{FuJi2020c}.

%%%%%%%%%%%%%%%%
\begin{lemma}\label{lemGam}
Assume that $\gamma$ satisfies \eqref{g1}. Then
	\begin{equation*}
	s\gamma(s)-a\gamma(a)\leq \Gamma(s),\qquad s\geq a.
	\end{equation*}
\end{lemma}
%%%%%%%%%%%%%%%%

\begin{proof}
	Since $\gamma$ is non-increasing, we infer that, for $s\geq a$,
	\begin{equation*}
	\Gamma(s)=\int_a^s\gamma(\eta)\mathrm{d}\eta\geq\gamma(s)(s-a)=s\gamma(s)-a\gamma(s).
	\end{equation*}
	Therefore, when $s\geq a$, using again the monotonicity of $\gamma$ which ensures that $\gamma(s)\leq \gamma(a)$,
	\begin{equation*}
	s\gamma(s)\leq \Gamma(s)+a\gamma(s)\leq \Gamma(s)+a\gamma(a),
	\end{equation*}
	as claimed.
\end{proof}

We next turn to an upper bound for $\Gamma$.

%%%%%%%%%%%%%%%%
\begin{lemma}\label{upG}
Under the assumptions \eqref{g1} and \eqref{gamma2}, there is $C>0$ depending on $a$ and $\gamma$ such that, for all $s\geq a$,
\begin{equation*}
\Gamma(s)\leq \Gamma_*(s)\triangleq\begin{cases}
C\log(s/a),\quad&\text{when}\;l=1\\
 \\
\displaystyle{\frac{C(s^{1-l}-a^{1-l})}{1-l}},\quad&\text{when}\;l\neq 1.
\end{cases}
\end{equation*}
\end{lemma}
%%%%%%%%%%%%%%%%

\begin{proof}
In view of assumption \eqref{gamma2}, there are $s_0>a$ and $C>0$ such that $s^l\gamma(s)\leq C$ for all $s\geq s_0$. In addition, the monotonicity of $\gamma$ and the non-negativity of $l$ ensure that $s^l\gamma(s)\leq s_0^l\gamma(a)$ for all $a\leq s\leq s_0$. Thus, $s^l\gamma(s)\leq C$ for all $s\geq a$, from which we deduce that
\begin{equation*}
\Gamma(s)=\int_a^s\gamma(\eta)\mathrm{d}\eta\leq C\int_a^s\eta^{-l}\mathrm{d}\eta=\begin{cases}
C\log(s/a),\quad&\text{when}\;l=1, \\
 \\
\displaystyle{\frac{C(s^{1-l}-a^{1-l})}{1-l}},\quad&\text{when}\;l\neq 1,
\end{cases}
\end{equation*}
as claimed.
\end{proof}

We finally recall the following lemma given in \cite[Lemma~A.1]{Laur1994}  which we shall use later in Section~\ref{sec3} to complete the Alikakos-Moser iterative argument.

%%%%%%%%%%%%%%%%
\begin{lemma}\label{lmiter}
	Let $\rho>1$, $b\geq0$, $c\in \mathbb{R}$, $C_0\geq1$, $C_1\geq1$, and $\delta_0$ be given numbers such that
	\begin{equation*}
		\delta_0+\frac{c}{\rho-1}>0.
	\end{equation*}
We consider the sequence $(\delta_j)_{j\geq0}$ of real numbers defined by
\begin{equation*}
	\delta_{j+1}=\rho\delta_j+c\,, \qquad j\in\mathbb{N}.
\end{equation*}
Assume further that $(\eta_j)_{j\geq0}$ is a sequence of positive real numbers satisfying 
	\begin{align*}
&\eta_0\leq C_1^{\delta_0},\\
&\eta_{j+1}\leq C_0\delta_{j+1}^{b}\max\{C_1^{\delta_{j+1}},\eta_j^\rho\}\,, \qquad j\in\mathbb{N}\,.
	\end{align*}	
Then the sequence $(\eta_j^{1/\delta_j})_{j\geq0}$ is bounded.
\end{lemma}
%%%%%%%%%%%%%%%%

%%%%%%%%%%%%%%%%
%%%%%%%%%%%%%%%%
\section{Global Existence of Classical Solutions to \eqref{cp}}\label{gecs}
%%%%%%%%%%%%%%%%
%%%%%%%%%%%%%%%%

In this section, we consider the existence of global classical solutions under the assumptions of Theorem~\ref{TH1}. First, we establish an upper bound for $v$. We next prove that $v$ is  H\"older continuous with respect to both $t$ and $x$, which finally allows us to apply the theory developed by Amann in \cite{Aman1990, Aman1995} for non-autonomous linear parabolic equations to derive $W^{2,p}$-estimates  for $v$.

From now on, we assume that $\gamma$ and $u^{in}$ satisfy \eqref{g0} and \eqref{ini}, respectively, and that $(u,v)$ is the corresponding solution to \eqref{cp} provided by Theorem~\ref{local} and defined on $[0,T_{\mathrm{max}})$. It then follows from \eqref{g0} and \eqref{e00} that 
	\begin{equation}
		\gamma(v(t,x))\le K_* \triangleq K_{v_*}\,, \qquad (t,x) \in [0,T_{\mathrm{max}})\times \bar{\Omega}\,.  \label{zz1}
	\end{equation}
Throughout this section, $C$ and $(C_i)_{i\ge 0}$ denote positive constants depending only on $\gamma$, $u^{in}$, $v_*$, and the parameter $a$ introduced in the definition of the indefinite integral $\Gamma$ of $\gamma$. Dependence upon additional parameters will be indicated explicitly.

%%%%%%%%%%%%%%%%
%%%%%%%%%%%%%%%%
\subsection{Time-dependent upper bounds}
%%%%%%%%%%%%%%%%
%%%%%%%%%%%%%%%%

First, we derive an upper bound for $v$ based on a comparison argument developed in \cite{FuJi2020,LiJi2020}.

%%%%%%%%%%%%%%%%
\begin{lemma}\label{tdub}
 The function $v$ satisfies the following key identity:
\begin{equation}\label{keyid}
\partial_t v +u\gamma(v)=(I-\Delta)^{-1}[u\gamma(v)]\;\;\text{ in }\;\; (0,T_{\mathrm{max}})\times \Omega.
\end{equation}
Here $\Delta$ denotes the usual Laplace operator supplemented with homogeneous Neumann boundary conditions.

Moreover, 
\begin{equation}\label{keyA}
v(t,x)\leq v^{in}(x) e^{tK_*}\,, \qquad (t,x) \in [0,T_{\mathrm{max}})\times \bar{\Omega}\,, 
\end{equation}
where $K_*$ is defined in \eqref{zz1} and $v^{in} {\modph =} (I-\Delta)^{-1}u^{in}\in W^{2,p}(\Omega)$ for all $p\in(1,\infty)$.
\end{lemma}
%%%%%%%%%%%%%%%%

\begin{proof}
First, the key identity \eqref{keyid} follows from applying $(I-\Delta)^{-1}$ to both sides of \eqref{cp1} and the property $(I-\Delta)^{-1}[u]=v$ due to \eqref{cp2}. Since $u$ and $\gamma$ are both non-negative, one obtains from the key identity \eqref{keyid} that $\partial_t v\leq (I-\Delta)^{-1}[u\gamma(v)]$ in $(0,T_{\mathrm{max}})\times \Omega$. Furthermore, due to assumption \eqref{g0}, we infer from the comparison principle for elliptic equations, \eqref{cp2},  and \eqref{zz1} that
\begin{equation*}
	(I-\Delta)^{-1}[u\gamma(v)]\leq (I-\Delta)^{-1}[K_* u]=K_* v\,.
\end{equation*} 
Hence $\partial_t v\leq K_* v$ in $(0,T_{\mathrm{max}})\times \Omega$, which gives rise to \eqref{keyA} by  direct integration. To complete the proof, we remark that, since $u^{in} \in C(\overline{\Omega})$, we have $v^{in}\in W^{2,p}(\Omega)$ for all  $p\in(1,\infty)$ by standard regularity theory for elliptic equations.
\end{proof}

%%%%%%%%%%%%%%%%
%%%%%%%%%%%%%%%%
\subsection{H\"older estimates for $v$}\label{sec.ir1}
%%%%%%%%%%%%%%%%
%%%%%%%%%%%%%%%%

The starting point is the key identity \eqref{keyid} which we combine with \eqref{cp2} to show that $v$ is the solution to a quasilinear parabolic equation in divergence form with right-hand side featuring a local quadratic nonlinearity $|\nabla v|^2$  and a nonlocal  nonlinear but regularizing operator. Specifically, by \eqref{cp2},
\begin{equation*}
u \gamma(v) = - \mathrm{div}(\gamma(v) \nabla v) + \gamma'(v) |\nabla v|^2 + v\gamma(v)\,. 
\end{equation*}
Therefore, introducing
\begin{equation}
f(s,\xi)\triangleq s \gamma(s) + \gamma'(s) |\xi|^2\,, \qquad (s,\xi)\in [0,\infty)\times \mathbb{R}^N\,, \label{rhsf}
\end{equation} and 
\begin{equation}
\varphi\triangleq(I-\Delta)^{-1}[u\gamma(v)]\,, \label{defvarphi}
\end{equation}
and recalling that 
\begin{equation*}
	\nabla \varphi \cdot \mathbf{n}  = 0\,, \qquad (t,x)\in (0,T_{\mathrm{max}})\times\partial\Omega\,,
\end{equation*}
since $\Delta$ denotes the usual Laplace operator supplemented with homogeneous Neumann boundary conditions, we conclude that $v$ solves the initial boundary value problem
\begin{subequations}\label{cpv}
	\begin{align}
	\partial_t v - \mathrm{div}(\gamma(v)\nabla v) & = -f(v,\nabla v) +\varphi\,, \qquad (t,x)\in (0,T_{\mathrm{max}})\times\Omega\,, \label{cpv1} \\
	\nabla v \cdot \mathbf{n} & = 0\,, \qquad (t,x)\in (0,T_{\mathrm{max}})\times\partial\Omega\,, \label{cpv2} \\
	v(0)  & = v^{in}\,, \qquad x\in\Omega\,. \label{cpv3}
	\end{align}
\end{subequations}
Now we fix $T\in(0,T_{\mathrm{max}})$ and set $J=[0,T]$. By  Lemma~\ref{tdub} and \eqref{e00}, we have
\begin{equation}
0 < v_* \le v(t,x) \le v^*(T)\triangleq \|v^{in}\|_{\infty}e^{TK_*}\;\;\text{ in }\;\; J\times \bar{\Omega}\,. \label{bv}
\end{equation}
Owing to the positivity and continuity of $\gamma$ on $(0,\infty)$, there is $\gamma^*(T)>0$ such that
\begin{equation}
\gamma(v)\ge \gamma^*(T) \;\;\text{ in }\;\; J\times \bar{\Omega}\,. \label{upv}
\end{equation} 
A straightforward consequence of \eqref{rhsf}, \eqref{bv}, and the regularity of $\gamma$ is that there are $C_0(T)>0$ and $C_1(T)>0$ such that
\begin{equation}
|f(v,\nabla v)| \le C_0(T) \left( 1 + |\nabla v|^2 \right) \;\;\text{in }\;\; J\times\bar{\Omega}\, \label{irb1} 
\end{equation}
and 
\begin{equation}
0\leq \varphi=(I-\Delta)^{-1}[u\gamma(v)]\leq K_* v\leq C_1(T) \triangleq K_* v^*(T) \;\;\text{in }\;\; J\times\bar{\Omega}.\label{phib}
\end{equation}

\medskip

We now proceed along the lines of \cite[Chapter~V, Section~7]{LSU1968} and \cite[Section~5]{Aman1989} to derive a local energy bound for $v$.

%%%%%%%%%%%%%%%%
\begin{lemma}\label{lem.impreg1}  There are $\delta(T)\in (0,1]$ and $C_2(T)>0$ such that, if $\vartheta\in C^\infty(J\times\bar{\Omega})$, $0\le \vartheta\le 1$, $\sigma\in\{-1,1\}$, and $h\in\mathbb{R}$ are such that
	\begin{equation}
	\sigma v(t,x) - h \le \delta(T)\,, \qquad (t,x)\in \mathrm{supp}\,\vartheta\,, \label{smallc1}
	\end{equation}
	then
	\begin{align*}
	& \int_\Omega \vartheta^2 (\sigma v(t) - h)_+^2\ \mathrm{d}x + \frac{\gamma^*(T)}{2} \int_{t_0}^t \int_\Omega \vartheta^2 |\nabla (\sigma v(\tau) - h)_+|^2\ \mathrm{d}x\mathrm{d}\tau \\
	& \qquad  \le \int_\Omega \vartheta^2 (\sigma v(t_0) - h)_+^2\ \mathrm{d}x + C_2(T) \int_{t_0}^t \int_\Omega  \left( |\nabla\vartheta|^2 + \vartheta|\partial_t\vartheta| \right) (\sigma v(\tau) - h)_+^2\ \mathrm{d}x\mathrm{d}\tau \\
	& \qquad\quad + C_2(T) \int_{t_0}^t  \int_{A_{h,\vartheta,\sigma}(\tau)} \vartheta\ \mathrm{d}x \mathrm{d}\tau 
	\end{align*}
	for $0\le t_0\le t\le T$, where 
	\begin{equation*}
	A_{h,\vartheta,\sigma}(\tau)\triangleq \left\{ x\in \Omega\ :\ \sigma v(\tau,x) > h \right\}\,, \qquad \tau\in [0,T_{\mathrm{max}})\,.
	\end{equation*}
\end{lemma}
%%%%%%%%%%%%%%%%

\begin{proof}
 Let $\delta\in (0,1]$ to be specified later. 
	It follows from \eqref{cpv1}, \eqref{cpv2}, and \eqref{irb1} that
	\begin{align*}
	\frac{1}{2} \frac{d}{dt} \int_\Omega \vartheta^2 (\sigma v - h)_+^2\ \mathrm{d}x & = \sigma \int_\Omega \vartheta^2 (\sigma v - h)_+ \partial_t v \ \mathrm{d}x + \int_\Omega  (\sigma v - h)_+^2 \vartheta \partial_t\vartheta\ \mathrm{d}x \\
	& = -\sigma \int_\Omega \gamma(v) \nabla \left[ \vartheta^2 (\sigma v - h)_+ \right]\cdot \nabla v\ \mathrm{d}x -\sigma \int_\Omega \vartheta^2 (\sigma v - h)_+ f(v,\nabla v)\ \mathrm{d}x \\
	& \quad + \sigma \int_\Omega \vartheta^2 (\sigma v - h)_+ \varphi\ \mathrm{d}x
	+ \int_\Omega (\sigma v - h)_+^2 \vartheta \partial_t\vartheta\ \mathrm{d}x \\
	& \le - \int_\Omega \gamma(v) \vartheta^2 |\nabla (\sigma v - h)_+|^2\ \mathrm{d}x + 2 \int_\Omega \gamma(v) \vartheta |\nabla\vartheta| (\sigma v - h)_+ |\nabla v|\ \mathrm{d}x \\
	& \quad + \int_\Omega (\sigma v - h)_+^2 \vartheta |\partial_t\vartheta|\ \mathrm{d}x + C_0(T) \int_\Omega \vartheta^2 (\sigma v - h)_+ \left( 1 + |\nabla v|^2 \right)\ \mathrm{d}x \\
	& \quad + \sigma\int_\Omega \vartheta^2 (\sigma v - h)_+\varphi \ \mathrm{d}x \\
	& = - \int_\Omega \gamma(v) \vartheta^2 |\nabla (\sigma v - h)_+|^2\ \mathrm{d}x + C_0(T) \int_\Omega \vartheta^2 (\sigma v - h)_+ |\nabla (\sigma v - h)_+|^2\ \mathrm{d}x \\
	& \quad + \int_\Omega (\sigma v - h)_+^2 \vartheta |\partial_t\vartheta|\ \mathrm{d}x + \sum_{k=1}^3 J_k\,,
	\end{align*}
	where
	\begin{align*}
	J_1 & \triangleq 2 \int_\Omega \gamma(v) \vartheta |\nabla\vartheta| (\sigma v - h)_+ |\nabla v|\ \mathrm{d}x\,, \\
	J_2 & \triangleq C_0(T) \int_\Omega \vartheta^2 (\sigma v - h)_+ \ \mathrm{d}x\,, \\
	J_3 & \triangleq \sigma\int_\Omega \vartheta^2 (\sigma v - h)_+ \varphi\ \mathrm{d}x \,.
	\end{align*}
	On the one hand, we infer from \eqref{g0}, \eqref{zz1}, and Young's inequality that 
	\begin{align*}
	J_1 & = 2 \int_\Omega \gamma(v) \vartheta |\nabla\vartheta| (\sigma v - h)_+ |\nabla (\sigma v - h)_+|\ \mathrm{d}x \\
	& \le \frac{1}{2} \int_\Omega \gamma(v) \vartheta^2 |\nabla (\sigma v - h)_+|^2\ \mathrm{d}x + 2 \int_\Omega \gamma(v) (\sigma v - h)_+^2 |\nabla\vartheta|^2\ \mathrm{d}x  \\
	& \le \frac{1}{2} \int_\Omega \gamma(v) \vartheta^2 |\nabla (\sigma v - h)_+|^2\ \mathrm{d}x + 2 K_* \int_\Omega  (\sigma v - h)_+^2 |\nabla\vartheta|^2\ \mathrm{d}x\,. 
	\end{align*}
	On the other hand, it follows from  \eqref{smallc1} that 
	\begin{equation*}
	J_2 \leq C_0(T) \delta \int_{A_{h,\vartheta,\sigma}}\vartheta\ \mathrm{d}x \leq C_0(T) \int_{A_{h,\vartheta,\sigma}}\vartheta\ \mathrm{d}x\,.
	\end{equation*}
Moreover, we observe that, if $\sigma=-1$, then $J_3\leq 0$ while, if $\sigma=1$, then we deduce from \eqref{phib} and \eqref{smallc1} that
\begin{equation*}
J_3 \le C_1(T) \delta\int_{A_{h,\vartheta,\sigma}}\vartheta \ \mathrm{d}x \leq C_1(T)\int_{A_{h,\vartheta,\sigma}}\vartheta\ \mathrm{d}x.
\end{equation*}	
	
Gathering the above inequalities gives
	\begin{align*}
	\frac{1}{2} \frac{d}{dt} \int_\Omega \vartheta^2 (\sigma v - h)_+^2\ \mathrm{d}x & \le -\frac{1}{2} \int_\Omega \gamma(v) \vartheta^2 |\nabla (\sigma v - h)_+|^2\ \mathrm{d}x \\
	& \quad + \int_\Omega  \left[ 2 K_* |\nabla\vartheta|^2 + \vartheta|\partial_t\vartheta| \right] (\sigma v - h)_+^2\ \mathrm{d}x \\
	& \quad + C_0(T) \int_\Omega \vartheta^2 (\sigma v - h)_+ |\nabla (\sigma v - h)_+|^2\ \mathrm{d}x \\
	& \quad + \left( C_0(T) + C_1(T) \right) \int_{A_{h,\vartheta,\sigma}} \vartheta\ \mathrm{d}x  \,. 
	\end{align*}
	Hence, owing to \eqref{bv}, \eqref{smallc1}, and the lower bound \eqref{upv} on $\gamma(v)$,
	\begin{align*}
	\frac{1}{2} \frac{d}{dt} \int_\Omega \vartheta^2 (\sigma v - h)_+^2\ \mathrm{d}x & \le \left( C_0(T) \delta - \frac{\gamma^*(T)}{2} \right) \int_\Omega \vartheta^2 |\nabla (\sigma v - h)_+|^2\ \mathrm{d}x \\
	& \quad + \int_\Omega  \left[ 2 K_* |\nabla\vartheta|^2 + \vartheta|\partial_t\vartheta| \right] (\sigma v - h)_+^2\ \mathrm{d}x \\
	& \quad + \left( C_0(T) + C_1(T) \right)  \int_{A_{h,\vartheta,\sigma}} \vartheta\ \mathrm{d}x \,. 
	\end{align*}
	Setting 
	\begin{equation*}
	\delta(T) \triangleq \min\left\{ 1 , \frac{\gamma^*(T)}{4C_0(T)} \right\} \;\;\text{ and }\;\; C_2(T) \triangleq 2 \max\left\{ 1 , 2 K_* , C_0(T) + C_1(T) \right\}\,,
	\end{equation*}
	and taking $\delta=\delta(T)$ in the above inequality, we end up with
	\begin{align*}
	& \frac{d}{dt} \int_\Omega \vartheta^2 (\sigma v - h)_+^2\ \mathrm{d}x + \frac{\gamma^*(T)}{2} \int_\Omega \vartheta^2 |\nabla (\sigma v - h)_+|^2\ \mathrm{d}x \\
	& \qquad  \le C_2(T) \left[ \int_\Omega  \left( |\nabla\vartheta|^2 + \vartheta|\partial_t\vartheta| \right) (\sigma v - h)_+^2\ \mathrm{d}x +  \int_{A_{h,\vartheta,\sigma}} \vartheta\ \mathrm{d}x \right] \,. 
	\end{align*}
	Integrating the above inequality with respect to time over $(t_0,t)$ completes the proof.
\end{proof}

We are now in a position to apply \cite[Chapter~II, Theorem~8.2]{LSU1968} to obtain a H\"older estimate for $v$.

%%%%%%%%%%%%%%%%
\begin{proposition}\label{prop.impreg2}
 There is $\alpha_T\in (0,1)$ depending on $N$, $\Omega$, $\gamma$, $v^{in}$, $v_*$, $a$, and $T$ such that $v\in C^{\alpha_T}(J\times\bar{\Omega})$.
\end{proposition}
%%%%%%%%%%%%%%%%

\begin{proof}
It follows from Lemma~\ref{lem.impreg1} that the estimate \cite[Chapter~II, Equation~(7.5)]{LSU1968} holds true (with parameters $q=r=2(1+\kappa)=(2N+4)/N$ satisfying \cite[Chapter~II, equation~(7.3)]{LSU1968}). Note that we do not require the test function $\vartheta$ in Lemma~\ref{lem.impreg1} to be compactly supported in $(0,T)\times \Omega$. Consequently, according to \cite[Chapter~II, Remark~7.2 and Remark~8.1]{LSU1968}, there is $C(T)$ such that $v$ belongs to the class $\hat{\mathcal{B}}_2([0,T]\times\bar{\Omega},v^*(T), C(T),(2N+4)/N,\delta(T),2/N)$. Taking also into account the assumed smoothness of the boundary of $\Omega$ and the H\"older continuity of $v^{in}\in C^{\alpha_0}(\bar{\Omega})$ for some $\alpha_0\in(0,1)$, we then infer from \cite[Chapter~II, Lemma~8.1 \& Theorem~8.2]{LSU1968} that $v\in C^{\alpha_T}([0,T]\times \bar{\Omega})$ for some $\alpha_T\in (0,1)$  depending on $T$, but also on other parameters as indicated in the statement of Proposition~\ref{prop.impreg2}.
\end{proof}

%%%%%%%%%%%%%%%%
%%%%%%%%%%%%%%%%
\subsection{$W^{2,p}$-estimates for $v$}\label{sec.ir2}
%%%%%%%%%%%%%%%%
%%%%%%%%%%%%%%%%

Thanks to the just derived H\"older estimates on $v$, we may now proceed as in \cite[Section~6]{Aman1989} to establish  bounds on the trajectory $\{v(t)\ :\ t\in [0,T_{\mathrm{max}})\}$ in appropriate Sobolev spaces which do not depend on $T_{\mathrm{max}}$. To this end, we proceed in two steps and begin with a bound in fractional Sobolev spaces.

%%%%%%%%%%%%%%%%
\begin{lemma}\label{lem.impreg3} Let  $T\in(0,T_{\mathrm{max}})$. For any	$p\in(1,\infty)$ and $\theta\in \left( \frac{1+p}{2p},1 \right)$, there is $C_3(T,p,\theta)>0$ such that
	\begin{equation*}
	\|v(t)\|_{W^{2\theta,p}} \le C_3(T,p,\theta)\,, \qquad t\in [0,T]\,. 
	\end{equation*}
\end{lemma}
%%%%%%%%%%%%%%%%

\begin{proof} We set $D_T \triangleq (v_*/2,2v^*(T))$ and $J \triangleq [0,T]$, where $v_*$ and $v^*(T)$ are defined in \eqref{e00} and \eqref{bv}, respectively. For $s\in D_T$, we define the elliptic operator $\mathcal{A}(s)$ by $\mathcal{A}(s)z \triangleq -\gamma(s) \Delta z$ and the boundary operator $\mathcal{B}z \triangleq \nabla z\cdot \mathbf{n}$. We point out that, since $\min_{D_T}\{\gamma\}>0$ due to \eqref{g0}, the operator $\mathcal{A}(s)$ satisfies \cite[Eq.~(4.6)]{Aman1990} for all $s\in D_T$ and is thus strongly uniformly elliptic for $s\in D_T$. Consequently, $(\mathcal{A}(s),\mathcal{B})$ is normally elliptic for all $s\in D_T$ by \cite[Theorem~4.2]{Aman1990}.  For $t\in J$, let $\tilde{\mathcal{A}}(t)$ be the $L^p$-realization of $(\mathcal{A}(v(t)),\mathcal{B})$ with domain 
\begin{equation*}
	W_{\mathcal{B}}^{2,p}(\Omega) \triangleq \{ z \in W^{2,p}(\Omega)\ :\ \nabla z\cdot \mathbf{n} = 0 \;\;\text{ on}\;\; \partial\Omega \}\,.
\end{equation*}
As in the proof of \cite[Theorem~6.1]{Aman1989}, the assumptions on $\gamma$ and Proposition~\ref{prop.impreg2} guarantee that 
\begin{equation*}
	\tilde{\mathcal{A}} \in BUC^{\alpha_T}(J, \mathcal{L}(W^{2,p}(\Omega),L^p(\Omega)))
\end{equation*}
and that $\tilde{\mathcal{A}}(J)$ is a \textit{regularly bounded subset} of $C^{\alpha_T}(J,\mathcal{H}(W^{2,p}(\Omega),L^p(\Omega)))$ in the sense of \cite[Section~4]{Aman1988} (or, equivalently, the condition \cite[(II.4.2.1)]{Aman1995} is satisfied). By \cite[Theorem~A.1]{Aman1990}, see also \cite[Theorem~II.4.4.1]{Aman1995}, there is a unique parabolic fundamental solution $\tilde{U}$ associated to $\{ \tilde{\mathcal{A}}(t)\ :\ t\in J \}$ and there exist positive constants $M_T$ and $\omega_T$ depending on $T$ such that 
\begin{equation}
	\|\tilde{U}(t,\tau)\|_{\mathcal{L}(W^{2,p}(\Omega))} + 	\|\tilde{U}(t,\tau)\|_{\mathcal{L}(L^p(\Omega))} + (t-\tau) 	\|\tilde{U}(t,\tau)\|_{\mathcal{L}(L^p(\Omega),W^{2,p}(\Omega))} \le M_T e^{\omega_T(t-\tau)} \label{irb2}
\end{equation}
for $0\le \tau < t\leq T$. Since $\theta\in \left( \frac{1+p}{2p},1 \right)$, it follows from \cite[Theorem~5.2]{Aman1993} that
	\begin{equation*}
	\left( L^p(\Omega) , W_{\mathcal{B}}^{2,p}(\Omega) \right)_{\theta,p} \doteq W_{\mathcal{B}}^{2\theta,p}(\Omega) \triangleq \{ z \in W^{2\theta,p}(\Omega)\ :\ \nabla z\cdot \mathbf{n} = 0 \;\;\text{ on}\;\; \partial\Omega \}\,,
	\end{equation*}
	and we infer from \cite[Lemma~II.5.1.3]{Aman1995} that there is $M_{\theta,T}>0$  depending on $T$ such that
	\begin{equation}
	\|\tilde{U}(t,\tau)\|_{\mathcal{L}(W_{\mathcal{B}}^{2\theta,p}(\Omega))} + (t-\tau)^\theta \|\tilde{U}(t,\tau)\|_{\mathcal{L}(L^p(\Omega),W_{\mathcal{B}}^{2\theta,p}(\Omega))} \le M_{\theta,T}  e^{\omega_T(t-\tau)} \label{irb3}
	\end{equation}
	for $0\le \tau<t\leq T$. We deduce from \eqref{cpv} that $v$ solves 
	\begin{equation}
	\begin{split}
	\partial_t v + \tilde{\mathcal{A}}(\cdot) v  & = F\,, \qquad t\in J\,, \\
	v(0) & = v^{in} \,, 
	\end{split}\label{irb4}
	\end{equation}
	where
	\begin{equation*}
	F =  \varphi - v\gamma(v)\,, \qquad (t,x)\in J\times \Omega\,. 
	\end{equation*}
	We recall that,  according to \eqref{zz1}, \eqref{bv}, and \eqref{phib}
		\begin{equation}
	\| F(t)\|_\infty \le C_4(T) \triangleq C_1(T) + v^*(T) K_*\,, \qquad t\in J\,, \label{irb5}
	\end{equation}
 	while the continuity of $u$ and $v$, see Theorem~\ref{local}, ensures that
	\begin{equation}
	F \in C(J\times\bar{\Omega})\,. \label{irb8}
	\end{equation}
	
Owing to Theorem~\ref{local},  \eqref{cpv}, and \eqref{irb8}, we observe that $v\in C(J, L^p(\Omega))\cap C^1((0,T], L^p(\Omega))$ is a solution to the linear initial-value problem \eqref{irb4} on $J$ in the sense of \cite[Section~II.1.2]{Aman1995} with $F\in C(J,L^p(\Omega))$ and $v(t)\in \mathrm{dom}(\tilde{\mathcal{A}}(t))$ for all $t\in (0,T]$. Since $\tilde{U}$ is a parabolic fundamental solution in the sense of \cite[Section~II.2.1]{Aman1995}, we are in a position to apply \cite[Remarks~II.2.1.2 (a)]{Aman1995} to conclude that $v$ has the representation formula
	\begin{equation}
	v(t) =  \tilde{U}(t,0) v^{in} + \int_0^t \tilde{U}(t,\tau) F(\tau)\ \mathrm{d}\tau\,, \qquad t\in J\,.\label{irb9}
	\end{equation}
	   We then infer from  \eqref{irb3}, \eqref{irb5}, and \eqref{irb9} that, for $t\in J $,
	\begin{align}
	\|v(t)\|_{W^{2\theta,p}} & \le M_{\theta,T} e^{\omega_{T} t} \|v^{in}\|_{W^{2\theta,p}} + M_{\theta,T} \int_0^t (t-\tau)^{-\theta} e^{\omega_{T} (t-\tau)} \|F(\tau)\|_p\ \mathrm{d}\tau \nonumber \\
	& \le M_{\theta,T} e^{\omega_{T} T} \|v^{in}\|_{W^{2,p}} + C_4(T) M_{\theta,T} |\Omega|^{1/p} \int_0^t (t-\tau)^{-\theta} e^{\omega_{T} (t-\tau)} \ \mathrm{d}\tau\,. \label{irb11}
	\end{align}
Since
	\begin{equation*}
\int_0^t (t-\tau)^{-\theta} e^{\omega_{T} (t-\tau)} \ \mathrm{d}\tau=	\int_0^t \tau^{-\theta} e^{\omega_{T} \tau}\ \mathrm{d}\tau \leq \frac{T^{1-\theta}e^{\omega_{T} T}}{1-\theta}\,,  \qquad t\in J\,,
	\end{equation*}
we conclude from \eqref{irb11} that
	\begin{equation*}
	\|v(t)\|_{W^{2\theta,p}} \le C(T,p,\theta)\,, \qquad t\in J\,,
	\end{equation*}
and the proof is complete.
\end{proof}

We now exploit further the parabolic feature of \eqref{cpv} to extend Lemma~\ref{lem.impreg3} to $\theta=1$.

%%%%%%%%%%%%%%%%
\begin{proposition}\label{prop.impreg4}
	Let $T\in(0, T_{\mathrm{max}})$. For any  $p\in (N,\infty)$,  there is $C_5(T,p)>0$ such that
	\begin{equation*}
	\|v(t)\|_{W^{2,p}} \le C_5(T,p)\,, \qquad t\in [0,T]\,. 
	\end{equation*}
\end{proposition}
%%%%%%%%%%%%%%%%

\begin{proof} We set $J=[0,T]$. Given $p\in (N,\infty)$, we may fix some $\theta\in ((N+p)/2p,1)$ and deduce from Lemma~\ref{lem.impreg3} that
	\begin{equation*}
		\|v(t)\|_{W^{2\theta,p}}\leq C(T,p)\,, \qquad t\in J\,.
	\end{equation*}
 Moreover, the choice of $\theta$ guarantees that $W^{2,p}(\Omega)$ is continuously embedded in $C^1(\bar{\Omega})$, so that the above estimate implies that
	\begin{equation}
	\|\nabla v(t)\|_\infty \le C(T,p)\,, \qquad t\in J\,. \label{irb14}
\end{equation}

Consider next $\xi\in (0,(p+1)/2p)$. It follows from \cite[Lemma~II.5.1.3]{Aman1995} that there is $L_{\xi,T}>0$ depending on $T$ such that
	\begin{equation}
	(t-\tau)^{1-\xi} \|\tilde{U}(t,\tau)\|_{\mathcal{L}(W^{2\xi,p}(\Omega),W^{2,p}(\Omega))} \le L_{\xi,T} e^{\omega_T(t-\tau)}\,, \qquad 0\le \tau<t\le T\,. \label{irb12}
	\end{equation} 
	We infer from \eqref{irb2}, \eqref{irb9}, and \eqref{irb12} that, for $t\in J$, 
	\begin{equation}
	\|v(t)\|_{W^{2,p}} \le M_T e^{\omega_T t} \|v^{in}\|_{W^{2,p}} + L_{\xi,T} \int_0^t (t-\tau)^{\xi-1} e^{\omega_T(t-\tau)} \|F(\tau)\|_{W^{2\xi,p}}\ \mathrm{d}\tau\,. \label{irb13}
	\end{equation}
	 Now, by direct calculations, we observe that
\begin{align*}
u \gamma(v) & = - \mathrm{div}(\gamma(v) \nabla v) + \gamma'(v) |\nabla v|^2 + v\gamma(v) \\
& = - \Delta \Gamma(v) + \Gamma(v) + v \gamma(v) - \Gamma(v) + \gamma'(v) |\nabla v|^2\\
& = - \Delta \Gamma(v) + \Gamma(v) +f(v,\nabla v) - \Gamma(v) \,. 
\end{align*}
Thus, we obtain that
\begin{equation}
\varphi=(I-\Delta)^{-1}[u\gamma(v)]=\Gamma(v)+(I-\Delta)^{-1}[f(v,\nabla v)-\Gamma(v)]\,. \label{nfvarphi}
\end{equation}
	
	Owing to the smoothing properties of $(I-\Delta)^{-1}$, we readily deduce from \eqref{g0}, \eqref{bv}, and \eqref{irb14} that, for $t\in J$, 
	\begin{align}
	\|(I-\Delta)^{-1}[f(v,\nabla v)-\Gamma(v)]\|_{W^{2\xi,p}} & \le \|(I-\Delta)^{-1} [f(v,\nabla v)-\Gamma(v)]\|_{W^{2,p}} \nonumber \\
	& \le C \|\gamma'(v)|\nabla v|^2+v\gamma(v)-\Gamma(v)\|_p \le C(T,p)\,. \label{irb15}
	\end{align}
	Also, the regularity of $\gamma$ ensures that the function $s\mapsto \Gamma(s)- s \gamma(s)$ obviously belongs to $W^{2,\infty}(D_T)$ (with $D_T=(v_*/2,2v^*(T))$) and we deduce from the continuous embedding of $W^{2\theta,p}(\Omega)$ in $W^{2\xi,p}(\Omega)$, Lemma~\ref{lem.impreg3}, and \cite[Theorem~2]{BrMi2001} that
	\begin{equation}
	\|\Gamma(v(t))-v(t) \gamma(v(t)) \|_{W^{2\xi,p}} \le C \left[ \|\Gamma(v(t))\|_{W^{2\theta,p}} + \|v\gamma(v)\|_{W^{2\theta,p}} \right] \le  C(T,p)\,, \qquad t\in J\,. \label{irb16}
	\end{equation}
 Since $F = \varphi - v \gamma(v)$ and $\varphi$ is also given by \eqref{nfvarphi}, we insert \eqref{irb15} and \eqref{irb16} in \eqref{irb13} to complete the proof, since $\xi-1>-1$. 
\end{proof}

\begin{proof}[Proof of Theorem~\ref{TH1}]
	With the aid of Proposition~\ref{prop.impreg4}, in the same way as done in \cite[Lemma~4.3]{AhnYoon2019}, we may further use a standard bootstrap argument to prove that, for any $0<T<T_{\mathrm{max}}$, there is $C=C(T)>0$ independent of $T_{\mathrm{max}}$ such that
\begin{equation*}
\sup\limits_{0\leq t\leq T}\|u(t)\|_{\infty}\leq C(T)\,.
\end{equation*}
According to Theorem~\ref{local}, we deduce that $T_{\mathrm{max}}= \infty$ and thus Theorem~\ref{TH1} is proved.
\end{proof}

%%%%%%%%%%%%%%%%
%%%%%%%%%%%%%%%%
\section{Uniform-in-time Boundedness of Classical Solutions to \eqref{cp}}\label{ubcs}
%%%%%%%%%%%%%%%%
%%%%%%%%%%%%%%%%

 The main outcome of Theorem~\ref{TH1} being the global existence of classical solutions to \eqref{cp} under the sole assumption \eqref{g0}, we are next interested in the uniform-in-time boundedness of these solutions under the assumptions of Theorem~\ref{TH2}.  We thus consider an initial condition $u^{in}$ satisfying \eqref{ini} and a motility function $\gamma$ satisfying \eqref{g1} and \eqref{gamma2} and first note that \eqref{g1} readily implies \eqref{g0} with $K_s = \gamma(s)$ for $s>0$. We then denote the unique global classical solution to \eqref{cp} provided by Theorem~\ref{TH1} by $(u,v)$ and recall that $v$ satisfies \eqref{e00} for all $(t,x)\in [0,\infty)\times \bar{\Omega}$. In this section, $C$, $(C_i)_{i\ge 0}$, and $(\lambda_i)_{i\ge 0}$ denote positive constants which depend only on $N$, $\Omega$, $\gamma$, $k$, $l$, $u^{in}$, $v_*$ and $a$. Dependence upon additional parameters will be indicated explicitly.

%%%%%%%%%%%%%%%%
%%%%%%%%%%%%%%%%
\subsection{Time-independent upper bounds}\label{sec3}
%%%%%%%%%%%%%%%%
%%%%%%%%%%%%%%%%

First, we derive a uniform-in-time upper bound  for $v$  by a Moser type iteration under the assumptions of Theorem~\ref{TH2}. More precisely, we prove the following result.

%%%%%%%%%%%%%%%%
\begin{proposition}\label{propunib}
 There is $v^*>0$ depending only on $N$, $\Omega$, $\gamma$, $u^{in}$, $k$, $l$, and $a$ such that
\begin{equation}
	\sup\limits_{t\geq0} \|v(t)\|_{\infty}\leq v^*\,.
\end{equation}
\end{proposition}
%%%%%%%%%%%%%%%%

 The proof of Proposition~\ref{propunib} consists of the following lemmas.

%%%%%%%%%%%%%%%%
\begin{lemma}\label{lem01}
 The function $v$ satisfies the following differential inequality
\begin{equation}
0 \le \partial_t v + u \gamma(v) \le \Gamma(v)+ a\gamma(a) \;\;\text{ in }\;\; (0,\infty)\times \Omega. \label{e1a}
\end{equation}
\end{lemma}
%%%%%%%%%%%%%%%%

We mention here that the proof of Lemma~\ref{lem01} only uses the assumption \eqref{g1} on $\gamma$.

\begin{proof} First, in view of the key identity \eqref{keyid},  we readily deduce from the non-negativity of $u$ and $\gamma$, the boundary conditions \eqref{cp3}, and the comparison principle that $\partial_t v +u\gamma(v)\ge 0$ in $(0,\infty)\times\Omega$.
	
 Next, recalling that $\varphi=(I-\Delta)^{-1}[u\gamma(v)]$, see \eqref{defvarphi}, we infer from \eqref{cp2} that
\begin{align*}
	\varphi-\Delta\varphi &= u\gamma(v) =(v-\Delta v)\gamma(v)\\
	&=-\Delta \Gamma(v)+\Gamma(v)+\gamma'(v)|\nabla v|^2+v\gamma(v)-\Gamma(v) \;\;\text{ in }\;\; (0,\infty)\times \Omega.
\end{align*}
In view of \eqref{g1}, \eqref{cp3}, and Lemma~\ref{lemGam}, we obtain that
\begin{equation*}
\varphi-\Delta \varphi\leq \Gamma(v)+ a\gamma(a) -\Delta (\Gamma(v)+a\gamma(a))\;\;\text{ in }\;\; (0,\infty)\times \Omega,
\end{equation*}
and $\nabla \varphi\cdot \mathbf{n} = \nabla (\Gamma(v) + a\gamma(a))\cdot \mathbf{n}=0$ on $(0,\infty)\times\partial\Omega$, which, according to the comparison principle, yields that
\begin{equation*}
	\varphi\leq \Gamma(v)+a\gamma(a) \;\;\text{ in }\;\; (0,\infty)\times \Omega.
\end{equation*}
 Combining the above inequality with the key identity \eqref{keyid} completes the proof.
\end{proof}

%%%%%%%%%%%%%%%%
\begin{lemma}\label{lmp} There are positive constants $\lambda_0>0$ and $C_0>0$ such that, for any $p>1+k$,
	\begin{equation*}
	\frac{d}{dt}\|v\|_p^{p}+\frac{\lambda_0 p(p-k-1)}{(p-k)^2} \|\nabla v^{\frac{p-k}{2}}\|_2^2+\lambda_0 p\| v\|_{p-k}^{p-k}\leq  C_0 p\int_\Omega\left( v^{p-1}\Gamma(v)+ v^{p-1}\right) \mathrm{d} x\,.
	\end{equation*}
\end{lemma}
%%%%%%%%%%%%%%%%

\begin{proof}
First, as in the proof of \cite[Lemma~3.3]{FuJi2020b}, assumption~\eqref{gamma2} guarantees that there exist  $b>0$ and $s_b>v_*$ depending only on $\gamma$ such that, for all $s\geq s_b$,
\begin{equation*}
1/\gamma(s)\leq bs^k\,,
\end{equation*}
while the monotonicity of $\gamma$ ensures that
\begin{equation*}
1/\gamma(s)\leq 1/\gamma(s_b)
\end{equation*}
for all $0 < s<s_b$. Therefore, 
\begin{equation}\label{cond_gamma}
1/\gamma(s)\leq bs^{k}+1/\gamma(s_b) \le C (1+s^k)\,, \qquad s > 0\,.
\end{equation}
Now, multiplying the inequality \eqref{e1a} by $v^{p-1}$ for some $p>1+k$, we obtain that
\begin{equation*}
\frac{1}{p}\frac{d}{dt}\|v\|_p^p+\int_\Omega u\gamma(v)v^{p-1}\ \mathrm{d} x \leq\int_\Omega v^{p-1}\Gamma(v)\ \mathrm{d} x + C \|v\|_{p-1}^{p-1}.
\end{equation*}	
Thanks to \eqref{g1}, \eqref{e00}, and \eqref{cond_gamma}, it follows that
\begin{align}
\int_\Omega u\gamma(v)v^{p-1}\ \mathrm{d} x & \geq C\int_\Omega (v^{k}+1)^{-1}v^{p-1}u\ \mathrm{d} x \ge C\int_\Omega (v^{k}+ v_*^{-k} v^k)^{-1}v^{p-1}u\ \mathrm{d} x \nonumber \\
& \ge C\int_\Omega v^{p-k-1}u\ \mathrm{d} x. \label{ile}
\end{align}
Next, recalling that $v-\Delta v=u$ by \eqref{cp2}, we observe that
\begin{align*}
\int_\Omega v^{p-k-1}u\ \mathrm{d} x & = \int_\Omega v^{p-k-1}(v-\Delta v)\ \mathrm{d} x = \|v\|_{p-k}^{p-k} + (p-k-1) \int_\Omega v^{p-k-2} |\nabla v|^2\ \mathrm{d} x \\
& =  \|v\|_{p-k}^{p-k} + \frac{4(p-k-1)}{(p-k)^2} \|\nabla v^{\frac{p-k}{2}}\|_2^2.
\end{align*}
Collecting the above estimates, we arrive at
\begin{equation*}
\frac{d}{dt}\|v\|_p^{p}+\frac{\lambda_0 p(p-k-1)}{(p-k)^2} \|\nabla v^{\frac{p-k}{2}}\|_2^2 +\lambda_0 p \|v\|_{p-k}^{p-k}\leq  C p\int_\Omega\left( v^{p-1}\Gamma(v)+ v^{p-1}\right)\ \mathrm{d} x,
\end{equation*}
and the proof is complete.
\end{proof}

%%%%%%%%%%%%%%%%
\begin{lemma}\label{lmrhs} 
 For any 
\begin{equation*}
p> \max\left\{\frac{2(N-1)}{N-2},l+\frac{N(k-l)}{2}\right\}
\end{equation*} 
and $\varepsilon>0$, there holds
\begin{equation}\label{rhs}
	\int_\Omega \left(v^{p-1}\Gamma(v)+ v^{p-1}\right)\mathrm{d} x\leq \varepsilon\|v^{\frac{p-k}{2}}\|_{H^1}^2 + C_1(\varepsilon,p)\,, \qquad t\ge 0\,.
\end{equation}
\end{lemma}
%%%%%%%%%%%%%%%%

\begin{proof} Recall that $N\ge 3$, $0\leq k<\frac{N}{N-2}$, and $0\leq k-l<\frac{2}{N-2}$. Observing that the choice of $p$ and $k$ guarantees that
\begin{equation*}
p\ge \frac{2(N-1)}{N-2} >\max\left\{1+\frac{N(k-1)_+}{2},1+k\right\},
\end{equation*} 
we split the argument into three cases. 
	
\noindent	Case~(I):  $1<l<\frac{N}{N-2}$. Thanks to Lemma~\ref{upG} and \eqref{e00},
\begin{equation*}
	\Gamma(v)\leq \frac{C(v^{1-l}-a^{1-l})}{1-l}=\frac{C}{l-1}(a^{1-l}-v^{1-l})\leq C,
\end{equation*}
so that
\begin{equation}
\int_\Omega \left(v^{p-1}\Gamma(v)+ v^{p-1}\right)\mathrm{d} x \le C \|v\|_{p-1}^{p-1}\,. \label{zz2}
\end{equation}
Now, since $k<\frac{N}{N-2}$, there is $q\in \left( 1 , \frac{N}{N-2} \right)$ such that
\begin{equation*}
q>\frac{(k-1)N}{2}\,. 
\end{equation*}
Since $q < \frac{N}{N-2} \le p-1< \frac{(p-k)N}{N-2}$ according to the choice of $p$ and $q$  (recall that $k\ge l>1$ here), it follows from H\"older's inequality and Corollary~\ref{cr3} that
\begin{equation*}
\|v\|_{p-1}\leq \|v\|^\alpha_{\frac{(p-k)N}{N-2}}\|v\|_{q}^{1-\alpha}\leq C(p) \|v\|^\alpha_{\frac{(p-k)N}{N-2}}
\end{equation*}
with 
\begin{equation*}
\alpha=\frac{N(p-k)(p-1-q)}{(p-1)[N(p-k)-q(N-2)]}\in (0,1).
\end{equation*}
Consequently,  recalling that $N\ge 3$, we deduce from Sobolev's inequality that
\begin{equation*}
\|v\|_{p-1}^{p-1}\leq C(p) \|v\|_{\frac{(p-k)N}{N-2}}^{(p-1)\alpha} = C(p) \|v^{\frac{p-k}{2}}\|_{\frac{2N}{N-2}}^{\frac{2(p-1)\alpha}{p-k}} \le C(p) \|v^{\frac{p-k}{2}}\|_{H^1}^{\frac{2(p-1)\alpha}{p-k}},
\end{equation*}
and one easily checks that the choice of $q$ implies that
\begin{equation*}
\frac{2(p-1)\alpha}{p-k}- 2 = \frac{2[(k-1)N-2q]}{N(p-k)-q(N-2)} < 0\,.
\end{equation*} 
Then \eqref{rhs} follows  from \eqref{zz2} and the above estimates by Young's inequality.

\smallskip

\noindent Case~(II): $l=1$. Fix $q\in \left( 1 , \frac{N}{N-2} \right)$ and $\delta\in (0,1)$ such that
\begin{equation*}
\delta < \min\left\{ \frac{2(p-1)}{N-2} , \frac{2q}{N} \right\}\,.
\end{equation*}
By Lemma~\ref{upG}, \eqref{e00}, and Young's inequality,
	\begin{equation*}
	\int_\Omega \left( v^{p-1}\Gamma(v) + v^{p-1} \right)\ \mathrm{d} x \leq C \int_\Omega v^{p-1} (1+\log v)\ \mathrm{d} x \leq C(p) \|v\|_{p-1+\delta}^{p-1+\delta} + C(p).	
	\end{equation*}
	
Since the choice of $p$ and $\delta$ ensures that $q < \frac{N}{N-2} < p-1+\delta<\frac{(p-1)N}{N-2}$, we infer from H\"older's inequality and Corollary~\ref{cr3} that
	\begin{equation*}
	\|v\|_{p-1+\delta}\leq \|v\|^\alpha_{\frac{(p-1)N}{N-2}}\|v\|_{q}^{1-\alpha}\leq C(p) \|v\|^\alpha_{\frac{(p-1)N}{N-2}}
	\end{equation*}
	with 
	\begin{equation*}
	\alpha=\frac{N(p-1)(p-1+\delta-q)}{(p-1+\delta)[N(p-1)-q(N-2)]}.
	\end{equation*}
Therefore, by Sobolev's inequality,
	\begin{equation*}
	\|v\|_{p-1+\delta}^{p-1+\delta} \leq C(p) \|v\|_{\frac{(p-1)N}{N-2}}^{(p-1+\delta)\alpha} = C(p) \|v^{\frac{p-1}{2}}\|_{\frac{2N}{N-2}}^{\frac{2(p-1+\delta)\alpha}{p-1}} \le C(p) \|v^{\frac{p-1}{2}}\|_{H^1}^{\frac{2(p-1+\delta)\alpha}{p-1}}.
	\end{equation*}
Owing to the choice of $q$ and $\delta$, we realize that
	\begin{equation*}
	\frac{2(p-1+\delta)\alpha}{p-1}- 2 = \frac{2(\delta N - 2q)}{N(p-1)-q(N-2)}<0\,,
	\end{equation*}
and \eqref{rhs} follows  the above estimates by Young's inequality.

\smallskip
	
\noindent	Case~(III): $0\leq l<1$. In this case, by \eqref{e00} and Lemma~\ref{upG},
\begin{equation*}
	\int_\Omega \left( v^{p-1}\Gamma(v) + v^{p-1} \right) \mathrm{d} x\leq \int_\Omega \left( C v^{p-l} + v_*^{l-1} v^{p-l} \right) \mathrm{d} x \le C \|v\|_{p-l}^{p-l}.
	\end{equation*}
Since $0\leq k-l<\frac{2}{N-2}$, we can pick $q\in \left( 1 , \frac{N}{N-2} \right)$ such that $q>\frac{N(k-l)}{2}$ and observe that the choice of $p$ ensures that 
\begin{equation*}
q < \frac{N}{N-2} \le p-1 < p-l <\frac{(p-k)N}{N-2}.
\end{equation*}	
We infer from H\"older's inequality and Corollary~\ref{cr3} that
	\begin{equation*}
	\|v\|_{p-l}\leq \|v\|^\alpha_{\frac{(p-k)N}{N-2}}\|v\|_{q}^{1-\alpha}\leq C(p) \|v\|^\alpha_{\frac{(p-k)N}{N-2}}
	\end{equation*}
	with 
	\begin{equation*}
	\alpha=\frac{N(p-k)(p-l-q)}{(p-l)[N(p-k)-q(N-2)]}\,.
	\end{equation*}
Thus, we deduce from Sobolev's inequality that
	\begin{equation*}
	\|v\|_{p-l}^{p-l}\leq C(p) \|v\|_{\frac{(p-k)N}{N-2}}^{(p-l)\alpha} = C(p) \|v^{\frac{p-k}{2}}\|_{\frac{2N}{N-2}}^{\frac{2(p-l)\alpha}{p-k}} = C(p) \|v^{\frac{p-k}{2}}\|_{H^1}^{\frac{2(p-l)\alpha}{p-k}}.
	\end{equation*}	
Due to the choice of $q$ and $p$, we find that
\begin{equation}
	\frac{2(p-l)\alpha}{p-k}-2=\frac{2[(k-l)N-2q]}{(p-k)N-(N-2)q}<0\,.
\end{equation}
 Once more, \eqref{rhs} follows from the above estimates by Young's inequality and the proof is complete.
\end{proof}

 We are now in a position to derive uniform-in-time estimates for $v$ in $L^p(\Omega)$ for any large enough finite value of $p$.

%%%%%%%%%%%%%%%%
\begin{lemma}\label{lmvp}
For any $p \geq 1$, there is $C_2(p)>0$ such that
	\begin{equation*}
		\|v(t)\|_p\leq C_2(p)\,, \qquad t\geq 0\,.
	\end{equation*}
\end{lemma}
%%%%%%%%%%%%%%%%

\begin{proof} 
We first observe that, since $\Omega$ is bounded, we may assume without loss of generality that
\begin{equation}
p> \max\left\{ 2+k, \frac{2(N-1)}{N-2},\frac{Nk}{2} \right\}\,. \label{plow}
\end{equation}	
Next, recalling Lemma~\ref{lmp}, we have
\begin{equation*}
\frac{d}{dt}\|v\|_p^{p}+\frac{\lambda_0 p(p-k-1)}{(p-k)^2} \|\nabla v^{\frac{p-k}{2}}\|_2^2+\lambda_0 p\|v\|_{p-k}^{p-k}\leq C_0 p\int_\Omega\left( v^{p-1}\Gamma(v)+ v^{p-1}\right)\ \mathrm{d} x.
\end{equation*}
Observing that $\|v\|_{p-k}^{p-k} = \|v^{\frac{p-k}{2}}\|_2^2$ and that, for $p\geq 2+k$,
\begin{equation*}
\frac{\lambda_0 p(p-k-1)}{(p-k)^2}>\lambda_0 \left( 1-\frac{1}{p-k} \right) \geq \frac{\lambda_0}{2},
\end{equation*} 
we deduce from the above inequality that
\begin{equation*}
\frac{d}{dt}\|v\|_p^p +\frac{\lambda_0}{2} \|v^{\frac{p-k}{2}}\|_{H^1}^2 \leq C_0 p\int_\Omega\left( v^{p-1}\Gamma(v)+ v^{p-1}\right) \mathrm{d} x.
\end{equation*}
 We now point out that $p>\frac{Nk}{2} \ge  l+\frac{(k-l)N}{2}$ due to $N\geq3$, $l\ge 0$, and the choice \eqref{plow} of $p$, so that we may apply Lemma~\ref{lmrhs} with an appropriate choice of $\varepsilon$, thereby obtaining
\begin{equation*}
\frac{d}{dt} \|v\|_p^p +\frac{\lambda_0 }{4} \|v^{\frac{p-k}{2}}\|_{H^1}^2 \leq C(p).
\end{equation*}
Since $p>\frac{Nk}{2}$, we infer from \eqref{e00} and Sobolev's inequality that
\begin{align*}
\|v\|_p^p & = \int_\Omega v^{\frac{N(p-k)}{N-2}} v^{\frac{kN-2p}{N-2}} \mathrm{d} x \le v_*^{\frac{kN-2p}{N-2}} \|v^{\frac{p-k}{2}}\|_{\frac{2N}{N-2}}^{\frac{2N}{N-2}} \le C(p) \|v^{\frac{p-k}{2}}\|_{H^1}^{\frac{2N}{N-2}} 
\end{align*}
and we finally arrive at
\begin{align*}
\frac{d}{dt}\|v\|_p^p +\lambda_1(p) \left( \|v\|_p^p \right)^{(N-2)/N}\leq C(p).
\end{align*}
Lemma~\ref{lmvp} then readily follows from the above differential inequality by the comparison principle for $p$ satisfying \eqref{plow}. We finally use the continuous embedding of $L^p(\Omega)$ in $L^q(\Omega)$ for $p\ge q$ to complete the proof.
\end{proof}

As in \cite{FuJi2020c, Jiang2020}, we next use Moser's iteration technique along the lines of \cite{Alik1979} to establish the boundedness of $v$ in  $L^\infty$ as stated in Proposition~\ref{propunib}.

\begin{proof}[Proof of Proposition~\ref{propunib}]
We recall that \eqref{e00}, \eqref{keyid}, the comparison principle, and the monotonicity of $\gamma$ imply that
\begin{equation*}
\partial_t v + u \gamma(v) \le \gamma(v_*) v \;\;\text{ in }\;\; (0,\infty)\times \Omega.
\end{equation*}
Then, for any $p \geq  2+k$, 
\begin{equation*}
\frac{d}{dt} \|v\|_p^p + p \int_\Omega u \gamma(v) v^{p-1} \mathrm{d} x \le p \gamma(v_*) \|v\|_p^p.
\end{equation*}
We also infer from \eqref{cp2}, \eqref{cp3}, and \eqref{ile} that 
\begin{align*}
\int_\Omega u \gamma(v) v^{p-1} \mathrm{d} x & \ge C \int_\Omega u v^{p-k-1} \mathrm{d} x = C \int_\Omega (v-\Delta v) v^{p-k-1} \mathrm{d} x \\
& = C \|v\|_{p-k}^{p-k} + 4C \frac{p-k-1}{(p-k)^2} \|\nabla v^{\frac{p-k}{2}} \|_2^2.
\end{align*}
Recalling that
\begin{equation*}
\frac{p(p-k-1)}{(p-k)^2}\ge \frac{1}{2}
\end{equation*}
for $p\ge 2+k$ and combining the above inequalities give
\begin{equation}
\frac{d}{dt} \|v\|_p^p + \lambda_2 \| v^{\frac{p-k}{2}} \|_{H^1}^2 \le p \gamma(v_*) \|v\|_p^p. \label{ami1}
\end{equation}

Next, by H\"older's and Sobolev's inequalities,
\begin{align*}
\|v\|_p^p & = \int_\Omega v^{\frac{p-k}{2}} v^{\frac{p+k}{2}} \mathrm{d} x \le \|v^{\frac{p-k}{2}}\|_{\frac{2N}{N-2}} \|v\|_{\frac{N(p+k)}{N+2}}^{\frac{p+k}{2}} \\
& \le C \|v^{\frac{p-k}{2}}\|_{H^1} \|v\|_{\frac{N(p+k)}{N+2}}^{\frac{p+k}{2}} .
\end{align*}
It then follows from Young's inequality that
\begin{equation*}
2p\gamma(v_*) \|v\|_p^p \le \lambda_2 \|v^{\frac{p-k}{2}}\|_{H^1}^2 + C p^2 \|v\|_{\frac{N(p+k)}{N+2}}^{p+k}. 
\end{equation*}
Combining \eqref{ami1} and the above estimate, we find
\begin{equation}
\frac{d}{dt} \|v\|_p^p + p\gamma(v_*) \|v\|_p^p \le 2p\gamma(v_*) \|v\|_p^p - \lambda_2 \| v^{\frac{p-k}{2}} \|_{H^1}^2 \le C p^2 \|v\|_{\frac{N(p+k)}{N+2}}^{p+k}. \label{ami2}
\end{equation}

Let us now define two sequences $(p_j)_{j\ge 0}$ and $(X_j)_{j\ge 0}$ by 
\begin{equation}
\begin{split}
p_{j+1} & = \frac{N+2}{N} p_j - k, \qquad j\ge 0, \qquad p_0 = \frac{kN}{2} + \frac{2(N-1)}{N-2}, \\ 
X_j & = \sup_{t\ge 0} \|v(t)\|_{p_j}^{p_j}, \qquad j\ge 0.
\end{split} \label{ami3}
\end{equation}
Note that $X_j<\infty$ for all $j\ge 0$ due to Lemma~\ref{lmvp} since the choice of $p_0$ guarantees that
\begin{equation}
p_{j+1} > p_j > p_0 > \frac{kN}{2}\,, \quad j\ge 0\,,  \;\;\text{ and }\;\; \lim_{j\to \infty} p_j = \infty. \label{ami4}
\end{equation} 
For $j\ge 0$, we take $p=p_{j+1}$ in \eqref{ami2} and use \eqref{ami3} to obtain
\begin{equation*}
\frac{d}{dt} \|v\|_{p_{j+1}}^{p_{j+1}} + \gamma(v_*) p_{j+1} \|v\|_{p_{j+1}}^{p_{j+1}} \le C p_{j+1}^2\|v\|_{p_j}^{\frac{(N+2) p_j}{N}} \le C p_{j+1}^2 X_j^{\frac{N+2}{N}}.
\end{equation*}
After an integration with respect to time, we find
\begin{align*}
\|v(t)\|_{p_{j+1}}^{p_{j+1}} & \le \|v^{in}\|_{p_{j+1}}^{p_{j+1}}  e^{-\gamma(v_*) p_{j+1} t} + C p_{j+1} X_j^{\frac{N+2}{N}} \left( 1 - e^{-\gamma(v_*) p_{j+1} t} \right) \\
& \le \max\left\{ |\Omega| \|v^{in}\|_\infty^{p_{j+1}} , C p_{j+1} X_j^{\frac{N+2}{N}} \right\}
\end{align*}
for all  $t\ge 0$, recalling that $v^{in} = (I-\Delta)^{-1} u^{in}$, see Lemma~\ref{tdub}. Hence,
\begin{equation*}
X_{j+1} \le C p_{j+1} \max\left\{ \|v^{in}\|_\infty^{p_{j+1}} , X_j^{\frac{N+2}{N}} \right\}, \qquad j\ge 0\,.
\end{equation*}
Owing to \eqref{ami4} and recalling that $X_{0}$ is finite by Lemma~\ref{lmvp},
we are in a position to apply Lemma~\ref{lmiter} with $\rho=\frac{N+2}{N}$, $b=1$, and $c=-k$ to conclude that the sequence $\left( X_j^{1/p_j} \right)_{j\ge 0}$ is bounded.  Since
\begin{equation*}
	\|v(t)\|_\infty = \lim_{j\to\infty} \|v(t)\|_{p_j}\,, \qquad t\ge 0\,, 
\end{equation*}
and $\|v(t)\|_{p_j} \le X_j^{1/p_j}$ for all $t\ge 0$ and $j\ge 0$, letting $j\to \infty$ completes the proof.
\end{proof}

%%%%%%%%%%%%%%%%
%%%%%%%%%%%%%%%%
\subsection{Improved regularity}\label{sec.ir}
%%%%%%%%%%%%%%%%
%%%%%%%%%%%%%%%%

With the time-independent upper bound on $v$ derived in Proposition~\ref{propunib}, we may argue in the same manner as in Section~\ref{gecs} to derive the uniform-in-time boundedness of $(u,v)$. More precisely, we regard $v$ as a solution to the following initial boundary value problem
\begin{subequations}\label{cpvu}
	\begin{align}
	\partial_t v - \nabla \cdot(\gamma(v)\nabla v) & = -f(v,\nabla v) +\varphi
	\,, \qquad (t,x)\in (0,\infty)\times\Omega\,, \label{cpvu1} \\
	\nabla v \cdot \mathbf{n} & = 0, \qquad (t,x)\in (0,\infty)\times\partial\Omega\,, \label{cpvu2} \\
	v(0)  & = v^{in}\,, \qquad x\in\Omega\,, \label{cpvu3}
	\end{align}
\end{subequations}
 the nonlinearity $f$ and the source term $\varphi$ being still defined by \eqref{rhsf} and \eqref{defvarphi}, respectively.
In view of  \eqref{e00} and Proposition~\ref{propunib},  we have
\begin{equation}
0 < v_* \le v(t,x) \le v^*\;\;\text{ in }\;\; [0,\infty)\times \bar{\Omega}\,, \label{bvu}
\end{equation}
so that, by the monotonicity of $\gamma$,
\begin{equation}
\gamma(v_*)\geq\gamma(v)\ge \gamma(v^*)>0 \;\;\text{ in }\;\; [0,\infty)\times \bar{\Omega}\,. \label{upvu}
\end{equation} 
A straightforward consequence of \eqref{rhsf}, \eqref{bvu}, and the regularity of $\gamma$ is that there is $C_3>0$ such that
\begin{equation}
|f(v,\nabla v)| \le C_3 \left( 1 + |\nabla v|^2 \right) \;\;\text{in }\;\; (0,\infty)\times\bar{\Omega}\,, \label{irb1u} 
\end{equation}
and 
\begin{equation}\label{phibu}
	0 \le \varphi=(I-\Delta)^{-1}[u\gamma(v)]\leq \gamma(v_*) v\leq v^* \gamma(v_*)\leq C_3\;\;\text{in }\;\; (0,\infty)\times\bar{\Omega}\,.
\end{equation}
 Thanks to these properties, we may then argue as in the proof of Proposition~\ref{prop.impreg2} to derive the H\"older continuity of $v$. We emphasize here that, owing to the uniform-in-time upper and lower bounds \eqref{bvu} on $v$, there is no time-dependence in the local energy estimates derived in Lemma~\ref{lem.impreg1}. Hence, using \cite[Chapter~II, Theorem~8.2]{LSU1968}, we derive a time-uniform version of Proposition~\ref{prop.impreg2}.

%%%%%%%%%%%%%%%%
\begin{proposition}\label{prop.impreg2u}
There is a time-independent constant $\alpha\in (0,1)$ such that $v\in BUC^{\alpha}([0,\infty),C^\alpha(\bar{\Omega}))$.
\end{proposition}
%%%%%%%%%%%%%%%%

As in Section~\ref{sec.ir2}, we now exploit the H\"older regularity on $v$ provided by Proposition~\ref{prop.impreg2u} and, as in the proofs of Lemma~\ref{lem.impreg3} and Proposition~\ref{prop.impreg4}, we proceed along the lines of \cite[Section~6]{Aman1989} to establish the boundedness of the trajectory $\{v(t)\ :\ t\ge 0\}$ in $W^{2,p}(\Omega)$ for all $p\in (N,\infty)$. We here take advantage of the validity of the needed bounds with constants which do not depend on time to derive estimates which also do not depend on time.

%%%%%%%%%%%%%%%%
\begin{proposition}\label{prop.impreg4u}
	For any  $p\in (N,\infty)$,  there is $C_4(p)>0$ such that
	\begin{equation*}
		\|v(t)\|_{W^{2,p}} \le C_4(p)\,, \qquad t\ge 0\,. 
	\end{equation*}
\end{proposition}
%%%%%%%%%%%%%%%%

\begin{proof} We first consider $\theta\in \left( \frac{1+p}{2p},1 \right)$. With the same notations as in Lemma~\ref{lem.impreg3}, we now have a unique parabolic fundamental solution $\tilde{U}$ associated to $\{ \tilde{\mathcal{A}}(t)\ :\ t\ge 0\}$ and there exist time-independent positive constants $M$, $M_\theta$, and $\omega$ such that 
	\begin{equation}
	\|\tilde{U}(t,\tau)\|_{\mathcal{L}(W^{2,p}(\Omega))} + 	\|\tilde{U}(t,\tau)\|_{\mathcal{L}(L^p(\Omega))} + (t-\tau) 	\|\tilde{U}(t,\tau)\|_{\mathcal{L}(L^p(\Omega),W^{2,p}(\Omega))} \le M e^{\omega(t-\tau)} \label{irb2u}
	\end{equation}
and
	\begin{equation}
	\|\tilde{U}(t,\tau)\|_{\mathcal{L}(W_{\mathcal{B}}^{2\theta,p}(\Omega))} + (t-\tau)^\theta \|\tilde{U}(t,\tau)\|_{\mathcal{L}(L^p(\Omega),W_{\mathcal{B}}^{2\theta,p}(\Omega))} \le M_\theta e^{\omega(t-\tau)} \label{irb3u}
	\end{equation}
	for $0\le \tau<t$. 
	
	We then pick $\mu>\omega$ and deduce from \eqref{cpvu} that $v$ solves 
	\begin{equation}
	\begin{split}
	\partial_t v + \left( \mu + \tilde{\mathcal{A}}(\cdot) \right)v  & = \tilde{F}\,, \qquad t>0\,, \\
	v(0) & = v^{in} \,, 
	\end{split}\label{irb4u}
	\end{equation}
	where
	\begin{equation*}
	\tilde{F} = F + \mu v = \mu v +\varphi - v\gamma(v) \,. 
	\end{equation*}
By the boundedness \eqref{bvu} of $v$ and \eqref{phibu}, there is $C_5>0$ such that
\begin{equation}
	\|\tilde{F}(t)\|_\infty\leq C_5\,, \qquad t\ge 0\,.
\end{equation}
 Owing the continuity of $u$ and $v$ provided by Theorem~\ref{local}, 
\begin{equation*}
	\tilde{F} \in C([0,\infty) \times\bar{\Omega})\,, 
\end{equation*}
and $v\in C([0,\infty),L^p(\Omega))\cap C^1((0,\infty),L^p(\Omega))$ is such that $v(t)\in \mathrm{dom}(\tilde{\mathcal{A}}(t))$ for all $t>0$. Consequently, $v$ is a solution to the linear initial-value problem~\eqref{irb4u} in the sense of \cite[Section~II.1.2]{Aman1995}. Using again \cite[Remarks~II.2.1.2~(a)]{Aman1995}, we conclude that $v$ has the representation formula 
	\begin{equation}
	v(t) = e^{-\mu t} \tilde{U}(t,0) v^{in} + \int_0^t e^{-\mu(t-\tau)} \tilde{U}(t,\tau) \tilde{F}(\tau)\ \mathrm{d}\tau\,, \qquad t\ge 0\,. \label{irb9u}
	\end{equation}
	We then infer from \eqref{irb3u} and \eqref{irb9u} that, for $t\ge 0$,
	\begin{align}
	\|v(t)\|_{W^{2\theta,p}} & \le M_\theta e^{(\omega-\mu)t} \|v^{in}\|_{W^{2\theta,p}} +M_\theta \int_0^t (t-\tau)^{-\theta} e^{(\omega-\mu)(t-\tau)} \|F(\tau)\|_p\ \mathrm{d}\tau  \nonumber \\
	& \le C(p,\theta) +  C_5 M_\theta |\Omega|^{1/p} \int_0^t (t-\tau)^{-\theta} e^{(\omega-\mu)(t-\tau)} \mathrm{d}\tau \nonumber \\
	&\leq C(p,\theta)\,,  \label{irb10u}
	\end{align}
since
	\begin{equation*}
	\mathcal{I}_\theta \triangleq \int_0^\infty \tau^{-\theta} e^{(\omega-\mu)\tau}\ \mathrm{d}\tau < \infty\,.
	\end{equation*}

 We next proceed as in the proof of Proposition~\ref{prop.impreg4} to extend the above estimate to $\theta=1$. To this end, given $\xi\in (0,(p+1)/2p)$, we observe that Proposition~\ref{prop.impreg2u} and \cite[Lemma~II.5.1.3]{Aman1995} imply that the bound \eqref{irb12} is valid for all $0\le \tau < t$, with a constant $L_\xi$ which depends, neither on $t$, nor on $\tau$. Then, thanks to \eqref{irb9u},
\begin{equation}
\|v(t)\|_{W^{2,p}} \le M e^{(\omega-\mu) t} \|v^{in}\|_{W^{2,p}} + L_\xi \int_0^t (t-\tau)^{\xi-1} e^{(\omega-\mu)(t-\tau)} \|\tilde{F}(\tau)\|_{W^{2\xi,p}}\ \mathrm{d}\tau\,, \label{irb13u}
\end{equation}
recalling that $\tilde{F}=F+\mu v$.  We then complete the proof  of Proposition~\ref{prop.impreg4u} with the help of \eqref{irb10u} and the corresponding time-independent version of estimates \eqref{irb15} and \eqref{irb16}, bearing in mind that $\mu>\omega$ and $\xi-1>-1$.
\end{proof}

 Theorem~\ref{TH2} is now a straightforward consequence of Proposition~\ref{prop.impreg4u} and a bootstrap argument.

\begin{proof}[Proof of Theorem~\ref{TH2}] 
With the aid of Proposition~\ref{prop.impreg4u}, we may further use a standard bootstrap argument (cf. \cite[Lemma~4.3]{AhnYoon2019}) to prove that
\begin{equation*}
\sup\limits_{t\ge 0}\left\{ \|u(t)\|_{\infty}+\|v(t)\|_{W^{1,\infty}} \right\}\leq C\,,
\end{equation*}
 and thus complete the proof.
\end{proof}

%%%%%%%%%%%%%%%%
%%%%%%%%%%%%%%%%
\section*{Acknowledgments}
J.~Jiang  is supported by Hubei Provincial Natural Science Foundation under the grant No. 2020CFB602. Ph.~Lauren\c{c}ot thanks Christoph Walker for illuminating discussions on the theory developed in \cite{Aman1995}. We also thank the referee for helpful remarks.
%%%%%%%%%%%%%%%%
%%%%%%%%%%%%%%%%

%%%%%%%%%%%%%%%%
%%%%%%%%%%%%%%%%
\bibliographystyle{siam}
\bibliography{Reference}
%%%%%%%%%%%%%%%%
%%%%%%%%%%%%%%%%

\end{document}